\documentclass[12pt]{article}

\usepackage[bookmarks={true},bookmarksopen={true},colorlinks,citecolor = green,anchorcolor = blue,linkcolor = blue]{hyperref}
\usepackage[utf8]{inputenc}
\usepackage[T1]{fontenc}
\usepackage[english]{babel}
\usepackage[dvipsnames]{xcolor}
\usepackage{lmodern}
\usepackage{array}
\usepackage{amsmath}
\usepackage{amssymb}
\usepackage{amsfonts}
\usepackage{amsthm}
\usepackage{mathrsfs}
\usepackage{enumitem}
\usepackage{bbm}
\usepackage{hhline}
\usepackage{fullpage}
\usepackage{tikz}
\usepackage{float}
\usepackage{authblk}

\usetikzlibrary[topaths]

\numberwithin{equation}{section}

\renewcommand{\leq}{\leqslant}
\renewcommand{\geq}{\geqslant}
\renewcommand{\tilde}{\widetilde}
\renewcommand{\epsilon}{\varepsilon}

\newcommand{\omoins}{\circleddash}

\newcommand{\der}{\mathrm{d}}
\newcommand{\E}{\mathbb{E}}

\newcommand{\K}{K}
\newcommand{\N}{\mathbb{N}}
\newcommand{\R}{\mathbb{R}}
\newcommand{\T}{\mathbb{T}}
\newcommand{\Z}{\mathbb{Z}}

\newcommand{\cara}[1]{\mathbbm{1}_{#1}}

\newcommand{\cvps}{\overset{a.s.}{\longrightarrow}}
\newcommand{\cvn}{\underset{n\rightarrow\infty}{\longrightarrow}}

\newcommand{\defeq}{:=}

\newcommand{\X}{\mathcal{X}}
\newcommand{\card}{\mbox{card}}

\newcommand{\tr}{\mbox{Tr}}

\newcommand{\rad}{\textrm{rad}}
\newcommand{\logdet}{\textnormal{logdet}}

\DeclareMathOperator*{\argmin}{arg\,min}
\DeclareMathOperator*{\argmax}{arg\,max}

\theoremstyle{plain}
\newtheorem{prop}{Proposition}[section]
\newtheorem{lem}[prop]{Lemma}
\newtheorem{defin}[prop]{Definition}
\newtheorem{defin/prop}[prop]{Definition/Proposition}
\newtheorem{theo}[prop]{Theorem}
\newtheorem{corr}[prop]{Corollary}

\title{Asymptotic approximation of the likelihood of stationary determinantal point processes}
\author[,1]{Arnaud Poinas\thanks{arnaud.poinas@math.univ-poitiers.fr}}
\author[,2]{Fr\'ed\'eric Lavancier\thanks{frederic.lavancier@univ-nantes.fr}}
\affil[1]{
Laboratoire de Mathématiques et Applications -- UMR 7348 -- Poitiers, France}
\affil[2]{Laboratoire de Math\'ematiques Jean Leray, Nantes, France. }

\date{\today}

\begin{document}
\maketitle
\begin{abstract}
Continuous determinantal point processes (DPPs)  are a class of repulsive point processes on $\R^d$ with many  statistical applications. Although an explicit expression of their density is known, it is too complicated to be used directly for maximum likelihood estimation. In the stationary case, an approximation using Fourier series has been suggested, but it is limited to rectangular observation windows and no theoretical results support it. In this contribution, we investigate a different way to approximate the likelihood by looking at its asymptotic behaviour when the observation window grows towards $\R^d$. This new approximation is not limited to rectangular windows, is faster to compute than the previous one,  does not require any tuning parameter,  and some  theoretical justifications are provided. It moreover provides an explicit formula for estimating the asymptotic variance of the associated estimator.
The performances are assessed in a simulation study on standard parametric models on $\R^d$ and compare favourably to common  alternative estimation methods for continuous DPPs.
\end{abstract}

\section{Introduction}

Determinantal point processes (DPPs for short) are a type of repulsive point processes with statistical applications ranging from machine learning \cite{Taskar} to telecommunications \cite{Telecom1,Telecom2,gomez_case_2015}, biology \cite{affandi2014learning}, forestry \cite{Lavancier}, signal processing \cite{STFT} and computational statistics \cite{BaHa16Sub}. In this paper, we focus on likelihood estimation of parametric families of stationary DPPs on $\R^d$, but we will also include in our study stationary DPPs defined on $\Z^d$.
From a theoretical point of view, we are specifically interested in an increasing domain setting, meaning that we assume to observe only one realization of the DPP on a bounded window $W\subset \R^d$, and our asymptotic results will concern the case where $W$ grows towards $\R^d$, making the cardinality of the observed DPP tend to infinity. From this perspective, the likelihood is just the density of the DPP. 

For a DPP on $\R^d$ with kernel $K$, the expression of its density on any compact set $W$ (with respect to the unit rate Poisson point process) is known since the seminal paper of Macchi (1975) \cite{Macchi}. But this expression is hardly tractable. It requires the knowledge of another kernel, usually called  $L$, that can only be obtained from $K$ by solving an integral equation or by knowing the spectral representation of the integral operator associated to $K$ on $W$. Some approximations are then needed. In the stationary case and when $W$ is a rectangular window, an approximation of the density has been proposed in \cite{Lavancier} by considering a Fourier series approximation of $K$. This approximation has the pleasant feature to be explicit, but is restricted to rectangular windows and lacks theoretical justifications. 
  
 Our contribution is an (increasing domain) asymptotic approximation  of the density as well as a way to correct the edge effects arising as a consequence of this approximation. This approach is not restricted to rectangular windows $W$, does not depend on any tuning parameter, and  is faster to compute than the Fourier series approximation of \cite{Lavancier}. 
Moreover, unlike the previous one, our approximation is generally smooth in the parameter of the model and we can compute explicitly its derivatives with respect to the parameter. This allows us to approximate the Fisher information matrix, and thus to estimate the asymptotic variance of the maximum likelihood estimator by an explicit formula.
 
The density of a DPP depends on the log-determinant of a random kernel matrix whose behaviour is difficult to control from a theoretical point of view, making challenging a theoretical study of our approximation. The situation simplifies slightly for stationary DPPs defined on a regular grid, typically $\Z^d$. We prove in this case that our approximation has the same asymptotic behaviour as the true density, under mild assumptions. The proof relies on an asymptotic control of the $L$ kernel when $W$ grows to $\Z^d$ and to concentration inequalities for DPPs established in \cite{Pemantle}. For DPPs defined on $\R^d$, getting the same kind of results remains an open problem. However we prove that any DPP on $\R^d$ is arbitrarily close to a DPP defined on a small enough regular grid, the density approximation of which is consistent from the previous result. 

Likelihood estimation of DPPs has been considered in other settings. For DPPs defined on a finite space, getting the expression of the density from $K$ is not an issue (providing the space dimension is not too large),  as it only requires the eigendecomposition of the kernel $K$, which reduces to a matrix in this case. Likelihood estimation in this setting, based on the observation of $n$ i.i.d. discrete DPPs, has been studied  in \cite{Brunel}, who investigate asymptotic properties when $n$ tends to infinity. In the continuous case, likelihood estimation based on $n$  i.i.d. observations is considered in \cite{Bardenet}. In this contribution, the DPP is directly defined through the kernel $L$, not $K$, avoiding the need to approximate its density from $K$. However, this comes at the cost of a loss of interpretability of the parameters, and more importantly, this approach does not allow to consider increasing domain asymptotic. Indeed, as detailed in Section~\ref{sec:Def}, it is extremely difficult to relate the kernel $L_{[W]}$ associated to the DPP defined on $W$, with the kernel $L_{[W']}$ for $W\subset W'$. For this reason, it is difficult to suggest a parametric family of kernels $L_{[W]}$ indexed by $W$.  In contrast the kernel $K$ of the DPP on any set $W$ is just the restriction of $K$ to $W$, and it suffices to define $K$ on $\R^d$ in order to automatically get a consistent family of kernels on any subset $W$. 

The remainder of the paper is organised as follows. We introduce our notations and basic definitions in Section~\ref{sec:Def}. Our asymptotic approximation of the likelihood  is presented in Section \ref{SEC:DEFLL}, along with some theoretical justifications. We show in Section~\ref{sec: appli} how this approximation applies to standard parametric families of DPPs in $\R^d$. Section~\ref{sec:Simu} is devoted to a simulation study demonstrating the performances of our approach. Some concluding remarks are given in Section~\ref{sec:Conclusion}. Finally Section~\ref{sec:ProofCons} includes the proof of our theoretical results, while some technical lemmas are gathered in the appendix.

\section{Definitions and notation} \label{sec:Def}
We consider point processes on $(\X,\mathcal{B}(\X),\nu)$ where 
$\X$ is either $\R^d$ or $\Z^d$ and the corresponding measure $\nu$ is either the Lebesgue measure on $\R^d$ or the counting measure on $\Z^d$, respectively. For any point process $X$ and $\nu$-measurable set $W\subset \X$ we write $N(W)$ for the number of points of $X\cap W$ and $|W|$ for the volume of $W$, i.e. $|W|=\nu(W)$ is either the Lebesgue measure of $W$ if $\X=\R^d$ or its cardinality if $\X=\Z^d$. Moreover, for any finite set $X\subset\X$ and any function $F:\X^2\rightarrow \R$, we write $F[X]$ for the matrix $(F(x,y))_{x,y\in X}$ where all $x\in X$ are arbitrarily ordered.  We write $F(x,y)\defeq F_0(y-x)$ if $F$ is invariant by translation, in which case  $F_0[X]$ will refer to the matrix $F[X]$,  and we write $F(x,y)\defeq F_{\rad}(\|y-x\|)$ if $F$ is a radial function. Here $\|.\|$ denotes the euclidean norm on $\X$ but we will also use the notation $\|.\|$ for the operator norm when applied to a linear operator, without ambiguity.  
We denote by $\hat f$ the Fourier transform of any function $f:\X\mapsto\R$, defined for any $x\in \X^*$ by
$$\hat f(x)\defeq\int_{\X}f(t)\exp(-2i\pi \langle t, x\rangle)\der\nu(t),$$
where $\X^*=\R^d$ if $\X=\R^d$, $\X^*=[0,1]^d$ if $\X=\Z^d$ and $\langle .,. \rangle$ denotes the usual scalar product on $\X$.
Finally, for any hermitian matrix $M$ we write $\lambda_{\max}(M)$ and $\lambda_{\min}(M)$ for the highest and the lowest eigenvalue of $M$, respectively, and, for any two hermitian matrices (or operators on a Hilbert space) $M$ and $M'$, we use the Loewner order notation $M\leq M'$ when $M'-M$ is positive definite.

\medskip

DPPs are commonly defined through their joint intensity functions.
\begin{defin} \label{Def1}
Let $X$ be a point process on $(\X,\nu)$ and $n\geq 1$ be an integer. If there exists a non-negative function $\rho_n:\X^n\rightarrow\R$ such that
$$\E\left [ \sum^{\neq}_{x_1,\cdots,x_n\in X} f(x_1,\cdots,x_n) \right ]=\int_{\X^n}f(x_1,\cdots,x_n)\rho_n(x_1,\cdots,x_n)\der\nu(x_1)\cdots\der\nu(x_n)$$
for all locally integrable functions $f:\X^n\rightarrow\R$, where the symbol $\neq$ means that the sum is done for distinct $x_i$, then $\rho_n$ is called the $n$-th order joint intensity function of $X$.
\end{defin}

DPPs are then defined the following way.
\begin{defin} \label{Def2}
Let $K:\X^2\rightarrow\mathbb{R}$  be a locally square integrable, hermitian function such that its associated integral operator on $L^2(\X,\nu)$,
$$\mathcal{K}:f\mapsto \left(\mathcal{K}f:x\mapsto \int_{\X}K(x,y)f(y)\der\nu(y)\right),$$
is locally of trace class with eigenvalues in $[0,1]$. $X$ is said to be a determinantal point process on $(\X,\mathcal{B}(\X),\nu)$ with kernel $K$ if its joint intensity functions exist and satisfy
\begin{equation}
\rho_n(x_1,\cdots,x_n)=\det\left(K[x]\right)
\end{equation}
for all integer $n$ and for all $x=(x_1,\cdots,x_n)\in\X^n$.
\end{defin}

When $\X=\R^d$ the DPP is said to be continuous and when $\X=\Z^d$ then the DPP is said to be discrete. In the latter case, the integral operator $\mathcal{K}$ can be seen as the infinite matrix $K[\Z^d]\defeq(K(x,y))_{x,y\in \Z^d}$. Moreover, when $K$ is translation invariant (resp. radial) then the associated DPP is stationary (resp. isotropic). Finally, we write $\mathcal{I}$ for the identity operator on $L^2(\X,\nu)$ and $\mathcal{I}_W$ for its restriction on $L^2(W,\nu)$ for any $W\subset\X$.

\bigskip

Let $X$ be a DPP on $\X$ with kernel $K$ and associated integral operator $\mathcal K$. If $\| \mathcal K\|<1$, then $X$ admits on any compact set $W\subset \X$ a density  with respect to the unit rate homogenous Poisson point process on $W$, as described now. We recall that for any compact set $W$, the projection $\mathcal{K}_W$ of $\mathcal{K}$ on $L^2(W,\nu)$ is a compact operator whose kernel can be written by Mercer's theorem as
$$K_W(x,y)=\sum_i \lambda^W_i\phi^W_i(x)\bar\phi^W_i(y),~\forall x,y\in W,$$
where the $\lambda_i^W$ are the eigenvalues of $\mathcal{K}_W$ and the $\phi^W_i$ are the corresponding family of orthonormal eigenfunctions (see \cite{Hough} for more details). When $\|\mathcal{K}\|<1$, we define the operator $\mathcal{L}=\mathcal{K}(\mathcal{I}-\mathcal{K})^{-1}$ and denote by $L$ its kernel. Similarly, we define the operator $\mathcal{L}_{[W]}=\mathcal{K}_W(\mathcal{I}_W-\mathcal{K}_W)^{-1}$  and denote by $L_{[W]}$ its kernel whose spectral decomposition reads
\begin{equation}\label{LWexpression}
L_{[W]}(x,y)\defeq\sum_i \frac{\lambda^W_i}{1-\lambda^W_i}\phi^W_i(x)\bar\phi^W_i(y).
\end{equation}
Note that contrary to $\mathcal K_W$ with $\mathcal K$,  the operator $\mathcal{L}_{[W]}$ does not correspond to the restriction of $\mathcal{L}$ to $L^2(W,\nu)$.  Another difference between $\mathcal{L}_{[W]}$ and $\mathcal{L}$ is that when $X$ is a stationary (resp. isotropic) DPP, $L(x,y)$ only depends on $y-x$ (resp. $\|y-x\|$) but this is not necessarily true for $L_{[W]}$.

\begin{theo}[\cite{Macchi, Sosh}] \label{Likelihood_creation}
Let $X$ be a DPP on $(\X,\nu)$ with kernel $K$ whose eigenvalues lie in $[0,1[$ and let $W$ be a compact set of $\X$. Then $X\cap W$ is absolutely continuous with respect to the homogeneous Poisson point process on $W$ with intensity $1$ and has density
$$f(x)=\exp(|W|)\det(\mathcal{I}_W-\mathcal{K}_W)\det(L_{[W]}[x])$$
for all $x\in\cup_n W^n$.
\end{theo}

In the above expression, the first determinant corresponds to the Fredholm determinant of the operator $\mathcal{I}_W-\mathcal{K}_W$, which is equal to $\prod_i (1-\lambda^W_i)$, while the second determinant is the standard matrix determinant.

\section{Likelihood of DPPs} \label{SEC:DEFLL}

\subsection{Likelihood estimation}

Let $X$ be a DPP on $(\X,\nu)$ with kernel $K^{\theta^*}$ belonging to a parametric family $\{K^\theta,\theta\in\Theta\}$, where $\Theta$ is the space of parameters. We consider the likelihood estimation of $\theta^*$, as described below, from a unique observation of $X\cap W$ where $W$ is a bounded subset of $\X$. We furthermore consider an increasing domain asymptotic framework, meaning that our asymptotic properties stand when $n\to\infty$ and $W=W_n$ is an increasing sequence of subsets of $\X$.

For the standard parametric families of continuous DPPs in $\R^d$, as those presented in Section~\ref{sec:ParFam}, the parameter space $\Theta$ is a  subset of $\R^p$ for some integer $p\geq 1$. However we do not need to make such an assumption for our purpose, and the likelihood approximation that  we develop below is true whatever $\Theta$ is, provided the associated DPP is stationary. In particular the parameter $\theta$ in $K^\theta$ can be the kernel $K$ itself.  This last setting makes sense when $\X=\Z^d$ where the whole matrix $K_W$ can be estimated from a realization of $X\cap W$, as considered in image analysis in \cite{launay2020}.

From Theorem~\ref{Likelihood_creation}, we get that the (normalized) log-likelihood of $X\cap W$ for any parametric family of DPPs reads:
\begin{equation}\label{LL}
l(\theta|X)=1+\frac{1}{|W|}\big(\logdet(\mathcal{I}_{W}-\mathcal{K}^\theta_{W})+\logdet(L^\theta_{[W]}[X\cap W])\big)
\end{equation}
where $\mathcal K^\theta$ is the integral operator associated to $K^\theta$ and $L^\theta_{[W]}$ is given by \eqref{LWexpression}, the eigenvalues and eigenvectors then depending on $\theta$. 
 The maximum likelihood estimate of $\theta$ is then
$$\hat\theta\in\argmax_{\theta\in\Theta} l(\theta|X).$$

Computing the log-likelihood \eqref{LL} requires knowing the spectral decomposition of $\mathcal{K}^\theta_{W}$ for all $\theta$. This is possible in the case of DPPs on a finite space whose kernels are finite matrices, provided the dimension of the space is not too large, but this spectral decomposition is usually not known for continuous DPPs. This motivates the following approximations. 

\subsection{Approximation of the likelihood for stationary DPPs}

When $\X=\R^d$ and the observation window $W$ is rectangular, an approximation of \eqref{LL} for stationary kernels is proposed in \cite{Lavancier}, using a truncated Fourier series. 
For example, if $W=[-l_1/2,l_1/2]\times\dots\times [-l_d/2,l_d/2]$ for some $l_1,\dots,l_d>0$,  denoting by $\Delta$ the diagonal matrix with diagonal entries $l_1,\dots,l_d$ (so that $\det(\Delta)=|W|$), 
this relies on the following approximation of the kernel:
\begin{equation}\label{approx1}
K^\theta(x,y)=K_0^\theta(x-y)\approx\sum_{\underset{\|k\|<N}{k\in\Z^d}} \frac{c_k}{|W|} e^{i2\pi\langle k, \Delta^{-1} (y-x)\rangle},
\end{equation}
where
\begin{equation}\label{approx2}
c_k\defeq \int_W K^\theta_0(t)e^{-i2\pi \langle k,\Delta^{-1}t\rangle}\der t\approx\hat K_0^\theta(\Delta^{-1} k),
\end{equation}
for some truncation constant $N$. Note that \eqref{approx1} is an equality if $(x-y)\in W$ and $N=\infty$, while  \eqref{approx2} is an equality when $K_0^\theta$ vanishes outside $W$.
Since the eigenvalues and eigenvectors of this kernel approximation are respectively $\hat K_0^\theta(\Delta^{-1} k)$ and $x\mapsto |W|^{-1/2}e^{i2\pi \langle k,\Delta^{-1}x\rangle}$, the log-likelihood \eqref{LL} is then approximated in \cite{Lavancier} by
\begin{equation}\label{eq:Fourier_approx}
 1  +\frac{1}{|W|}\left(\sum_{\underset{\|k\|<N}{k\in\Z^d}}\log(1-\hat K_0^\theta(\Delta^{-1} k))+\logdet(L^\theta_{app}[X\cap W])\right),
 \end{equation}
 where 
\begin{equation}\label{Lapp}
L^\theta_{app}(x,y)\defeq \frac 1 {|W|} \sum_{\underset{\|k\|<N}{k\in\Z^d}}\frac{\hat K_0^\theta(\Delta^{-1}k)}{1-\hat K_0^\theta(\Delta^{-1}k)}e^{i2\pi  \langle k,\Delta^{-1}(y-x)\rangle}.
\end{equation}
The same kind of approximations can be carried out when $\X=\Z^d$, still for rectangular windows $W$, in which case $\hat K_0^\theta$ in \eqref{eq:Fourier_approx} has to be replaced by the discrete Fourier transform of $K_0(x)$, $x\in W$, and no truncation is needed since the series become a finite sum. This approximation in $\Z^d$ amounts to consider a periodic extension of the stationary DPP outside $W$, see  \cite{launay2020} for details.

Our new approximation is based on a different expression of \eqref{LL} in terms of the self-convolution products of the function $(x,y)\mapsto\cara{W}(x)K^\theta(x,y)\cara{W}(y)$ through the following identities (see \cite{Shirai} for example). For all $W\subset\X$,
\begin{equation} \label{convol1}
\logdet(\mathcal{I}_{W}-\mathcal{K}^\theta_{W})=-\sum_{k=1}^{\infty}\frac{1}{k}\int_{W^k}K^\theta(x_1,x_2)\cdots K^\theta(x_{k-1},x_k)K^\theta(x_k,x_1)\der\nu^k(x)
\end{equation}
and for all $x,y\in W$,
\begin{align}
L^\theta_{[W]}(x,y) &=K^\theta(x,y)+\sum_{k=1}^{\infty}\int_{W^k}K^\theta(x,z_1)K^\theta(z_1,z_2)\cdots K^\theta(z_{k-1},z_k)K^\theta(z_k,y)\der\nu^k(z),\\
L^\theta(x,y)& =K^\theta(x,y)+\sum_{k=1}^{\infty}\int_{\X^k}K^\theta(x,z_1)K^\theta(z_1,z_2)\cdots K^\theta(z_{k-1},z_k)K^\theta(z_k,y)\der\nu^k(z).  \label{convol2}
\end{align}

These convolution products are too difficult to be computed in the general case, but for stationary DPPs satisfying $\|\mathcal{K^\theta}\|<1$ then $\hat L^\theta_0=\hat K^\theta_0/(1-\hat K^\theta_0)$ as a consequence of \eqref{convol2}. Accordingly,  as justified later in Proposition~\ref{DeterministicCV}, an asymptotic approximation when the observation window $W$ is large enough  gives 
\begin{align}
&L^\theta_{[W]}(x,y)\approx L^\theta_0(y-x)=\int_{\X^*}\frac{\hat K^\theta_0(t)}{1-\hat K^\theta_0(t)}\exp(2i\pi \langle t , y-x\rangle)\der t,\label{app1}\\
&\frac{1}{|W|}\logdet(\mathcal{I}_{W}-\mathcal{K}^\theta_{W})\approx\int_{\X^*}\log(1-\hat K^\theta_0(x))\der x.\label{app2}
\end{align}
This motivates our following approximation of the log-likelihood:
\begin{equation} \label{approxLL}
\tilde l(\theta|X)\defeq 1+\int_{\X^*}\log(1-\hat K^\theta_0(x))\der x+\frac{1}{|W|}\logdet(L^\theta_0[X\cap W]),
\end{equation}
where $L_0^\theta$ is given in \eqref{app1}. This approximation, like \eqref{eq:Fourier_approx},  can be computed whenever we know the expression of $\hat K^\theta_0$, which is the case for all classical families of stationary DPPs built from covariance functions, as those presented in Section \ref{sec:ParFam}. The main advantage of \eqref{approxLL} compared to the Fourier approximation \eqref{eq:Fourier_approx} is that it is not limited to rectangular windows $W$ but can be used with any window shape. It has also the advantage of not requiring any tuning  parameter of any kind compared to the choice of $N$ in \eqref{eq:Fourier_approx} or alternative moment methods \cite{biscio:lavancier:17,lavancier2019adaptive}. 

\bigskip

The idea to use a convolution approximation was actually briefly suggested in \cite[Appendix L]{Lavancier} but the  associated approximation was given under a more restrictive form that required knowing an exact expression of the iterative self-convolution products of $K^\theta_0$ for all $\theta$. Moreover, an important drawback was pointed out in \cite{Lavancier} concerning the presence of possible edge effects, which may affect the quality of estimation of strongly repulsive DPPs. As shown in Section \ref{sec:Simu}, this problem also occurs with our approximation: while it works really well with DPPs with low repulsion, and therefore minimal edge effects, some edge corrections are needed for more repulsive DPPs. The next section deals with this aspect.

\subsection{Periodic edge-corrections} \label{sec:periodic_correction}

In order to alleviate the possible edge-effects mentioned above, we suggest to introduce a periodic approximation. 
We assume in this section that the observation window $W\subset\X$ is rectangular. Without loss of generality, we set $W=\left([-l_1/2,l_1/2]\times\cdots\times[-l_d/2,l_d/2]\right)\cap \X$. Using a periodic approximation amounts to consider the observation window as the flat torus $\T_W\defeq \X\backslash l_1\Z\times\cdots\times\X\backslash l_d\Z$.  This way, points close to the border of the window $W$ are brought close to each other in order to compensate edge effects.

More precisely, we replace all instances of $L_0^\theta(y-x)$  in the stochastic part $L^\theta_0[X\cap W]$ of \eqref{approxLL} by 
$$L^{\theta,\mathbb{T}}_0(y-x)\defeq L_0^\theta
\begin{pmatrix}
y_1-x_1~\textrm{mod}(l_1) \\
\vdots \\
y_d-x_d~\textrm{mod}(l_d)
\end{pmatrix}.$$
This is equivalent to replacing $L_0^\theta$ by a periodic version of itself on $W$. The approximate likelihood then reads for any parameter $\theta$:

\begin{equation} \label{approxLLtorus}
\tilde l^{~\T}(\theta|X)\defeq 1+\int_{\X^*}\log(1-\hat K_0^\theta(x))\der x+\frac{1}{|W|}\logdet\left(L^{\theta,\mathbb{T}}_0[X\cap W]\right).
\end{equation}

Note that since we consider a periodic version of $L_0^\theta$ on $W$ then it can be approximated by its Fourier series, which corresponds to the idea of the approximation \eqref{eq:Fourier_approx} of \cite{Lavancier}. This is why both \eqref{approxLLtorus} and \eqref{eq:Fourier_approx} are nearly equal, see Figure~\ref{ComparingSmoothness} for an example. But approximating $L^\theta_{[W]}$ as in \eqref{approxLLtorus} instead of using a truncation of its Fourier series leads to a smoother likelihood and overall slightly better results, as well as a more computationally efficient method. Indeed, as explained in \cite{Lavancier}, the Fourier approximation \eqref{Lapp} of $L^\theta_{[W]}$ is a sum of $(2N)^d$ terms where the truncation parameter $N$ is chosen such that 
$$\sum_{n\in\Z^d\cap [-N,N]^d} \hat K_0^\theta(n)>0.99\sum_{n\in\Z^d} \hat K_0^\theta(n).$$
For important parametric models, including the Whittle-Matern and the Bessel families (see Section~\ref{sec:ParFam}), $\hat K_0^\theta(n)$ has a polynomial decay with respect to $n$, leading to a large choice of $N$ in \eqref{Lapp}. In comparison, as detailed in section \ref{sec:L0}, depending on the parametric model, we either have an analytic expression of $L_0^\theta$ or, when the self convolution products of $K_0^\theta$ are known, we can express $L_0^\theta$ as the infinite sum
\begin{equation} \label{eq:L0_infinite_sum}
L_0^\theta(x)=\sum_{n\geq 1}{(K_0^\theta)}^{*n}(x)
\end{equation}
where
$$|(K_0^\theta)^{*n}(x)|=\left|\int_{\X^*} (\hat {K^\theta_0})^n(t)e^{-i2\pi\langle x,t\rangle}\der t\right|\leq K_0^\theta(0)\|\hat K_0^\theta\|_{\infty}^{n-1}$$
has an exponential decay with respect to $n$. The approximation of $L_{[W]}^\theta$ by \eqref{eq:L0_infinite_sum} will then require much fewer terms than the approximation by \eqref{Lapp}.

\begin{figure}
\centering
   \includegraphics[width=10cm]{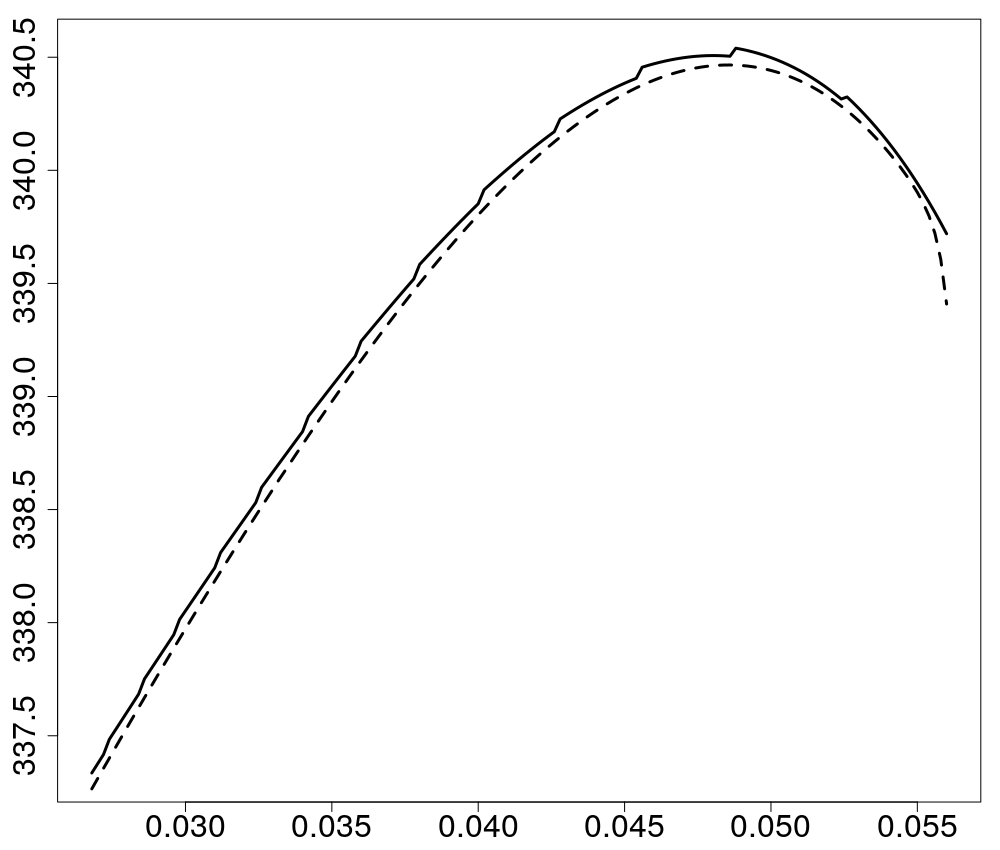}
	   \caption{\label{ComparingSmoothness} Comparison between the two approximations \eqref{eq:Fourier_approx} (solide line) and \eqref{approxLLtorus} (dashed line) of $\alpha\mapsto l(\rho^*,\alpha|X)$ where $X$ is a realization of a DPP with Gaussian-type kernel (see Table~\ref{parametric models})  with true parameters $\rho^*=100$ and $\alpha^*=0.05$ on the window $W=[0,1]^2$.}
\end{figure}

Despite the appealing of the approximation \eqref{approxLLtorus}, there is one possible issue in that  the determinant of $L^{\theta,\mathbb{T}}[X\cap W]$ is not necessarily positive. Remember that this positivity is guaranteed for any $X\cap W$ whenever the kernel $L^{\theta,\mathbb{T}}$ is positive, or equivalently whenever  its associated  integral operator has positive eigenvalues. But due to the periodicity of $L_0^{\theta,\mathbb{T}}$, these eigenvalues correspond to the coefficients of its Fourier series that read for any $k=(k_1,\dots,k_d)$
\begin{equation}\label{EVperiod}
\frac{1}{|W|}\int_W L^\theta_0(x)\exp\left(-2i\pi\sum_{i=1}^d \frac{k_ix_i}{l_i}\right)\der\nu(x).
\end{equation}
When $W$ is large, the above integral is approximately equal to $\hat L^\theta_0(k_1/l_1,\cdots,k_d/l_d)$ which is positive.  This shows that we can expect the determinant of $L^\mathbb{\theta, T}[X\cap W]$ to be positive when $W$ is large enough. In our simulations displayed in Section~ \ref{sec:Simu}, this determinant was positive in all runs, except a few times with the Bessel-type kernel associated to high values of the repulsion parameter $\alpha$.

Finally, note that extending the above edge correction to non rectangular windows is not straightforward and we do not provide a general solution. We however introduce in the simulation example of Section~\ref{sec: simuR} a procedure that can be adapted to any isotropic DPP model.

\subsection{Theoretical Results}\label{section theory}
In order to verify the theoretical soundness of the asymptotic log-likelihood approximation \eqref{approxLL} we want to show that $|\tilde l(\theta|X)-l(\theta|X)|\cvps 0$ uniformly for all $\theta\in\Theta$ when the observation window $W$ grows towards $\X$. 
For this purpose, we consider a sequence of increasing observation windows $W_n$ satisfying the following assumptions.

\medskip

\noindent {\bf Condition $(\mathcal{W})$}: $W_n$ is an increasing sequence of compact subsets of $\X$ such that $\bigcup_{n\geq 0}W_n=\X$ and there exists an increasing non-negative sequence $r_n\in\R_+^\N$ such that $r_n\cvn \infty$ and
\begin{equation} \label{CondWindow}
|(\partial W_n\oplus r_n)\cap W_n|=o(|W_n|),
\end{equation}
where, by a small abuse of notation, we write $\partial W_n\oplus r_n$ for the Minkowski sum of $\partial W_n$, the boundary of $W_n$, and a centered ball with radius $r_n$, which corresponds to the set of points whose distance to the boundary of $W_n$ is lower than $r_n$.
Moreover,
\begin{equation} \label{notslow}
\forall\delta>0,~\sum_{n\geq 0}\exp(-\delta|W_n|)<\infty.
\end{equation}

\medskip

The first assumption \eqref{CondWindow} means that the boundary of $W_n$ must not be too irregular. This is not an issue in most practical applications. For example, if $\X=\R^d$ and $(W_n)_{n\geq 0}$ is a sequence of spheres with radius $R_n\cvn\infty$, then \eqref{CondWindow} is satisfied with $r_n=\sqrt{R_n}$.
As another example, assume that $(W_n)_{n\geq 0}$ is a sequence of rectangular windows $[-l_{1,n}/2,l_{1,n}/2]\times\cdots\times[-l_{d,n}/2,l_{d,n}/2]$ such that $l_{i,n}\cvn\infty$ for each $i$, then 
\begin{multline*}
(\partial W_n\oplus r_n)\cap W_n\subset\big([-l_{1,n}/2, -l_{1,n}/2+r_n]\cup[l_{1,n}/2-r_n,l_{1,n}/2]\big)\times\cdots\\
\times\big([-l_{d,n}/2, -l_{d,n}/2+r_n]\cup[l_{d,n}/2-r_n,l_{d,n}/2]\big)
\end{multline*}
hence
$$\frac{|(\partial W_n\oplus r_n)\cap W_n|}{|W_n|}\leq\prod_{i=1}^d\left(\frac{2r_n}{l_{i,n}}\right)$$
which vanishes when $n$ goes to infinity with the choice $r_n=\sqrt{\min_i l_{i,n}}$. The second hypothesis \eqref{notslow} is a technical assumption needed to get the almost sure convergence in Proposition~\ref{approxLLCV}. Without this assumption, the convergence remains true but in probability instead of  almost surely.

\bigskip

We first consider the uniform convergence of the deterministic part of \eqref{LL}, which is the Fredholm log-determinant. Its asymptotic behaviour given below is justified  in Section~\ref{DetPart} and was already proved in a slightly different setting in \cite[Proposition 5.9]{Shirai}.
\begin{prop} \label{DeterministicCV}
Let $\{K^\theta_0:\X\rightarrow\R, ~\theta\in\Theta\}$ be a family of functions in $L^1(\X,\nu)$ with integrable Fourier transforms $\hat K^{\theta}_0$ taking values in $[0,M]$ for some $M<1$ and let $(W_n)_{n\geq 0}$ satisfy Condition $(\mathcal{W})$. Additionally, we assume that $\sup_{\theta\in\Theta}K^{\theta}(0)<\infty$ and that the function $x\mapsto\sup_{\theta\in\Theta}|K^{\theta}_0(x)|$ is integrable on $(\X,\nu)$. We denote by $\mathcal{K}^{\theta}_{W_n}$ the projection on $L^2(W_n)$ of the integral operator associated with the kernel $(x,y)\mapsto K^{\theta}_0(x-y)$. Then,
$$\sup_{\theta\in\Theta}\left|\frac{1}{|W_n|}\logdet(\mathcal{I}_{W_n}-\mathcal{K}^{\theta}_{W_n})-\int_{\X^*}\log(1-\hat K^{\theta}_0(x))\der x\right|\cvn 0.$$
\end{prop}
Concerning the stochastic part of the log-likelihood \eqref{LL}, that is  $\logdet(L^\theta_{[W]}[X\cap W])$, its  behaviour is much more difficult to control in general. The main issue is that the determinant vanishes when two points of $X\cap W$ gets arbitrarily close to each other, but no relationship between how close these points are from each other and the value of the determinant is known, making the likelihood difficult to control. To our knowledge, the only related result is that, in most cases, the lowest eigenvalue of $L^\theta_{[W]}[X]$ is non zero iff $\inf_{x,y\in X}\|y-x\|>0$ \cite{min}. The latter condition is automatically satisfied if $X$ is supported on a lattice but not when $\X=\R^d$. The next result focuses on the first case. 

\begin{prop} \label{approxLLCV}
Let $(W_n)_{n\in\N}$ satisfy Condition $(\mathcal{W})$ and let $\{K^\theta,\theta\in\Theta\}$ be a family of translation-invariant DPP kernels on $\Z^d$ such that $\Theta$ is a compact set of $\R^p$ for some integer $p\geq 1$ and the function $(\theta,x)\mapsto \hat K^\theta(x)$ is continuous on $\Theta\times\R^d$. Additionally, assume there exists constants $A,\tau>0$ and $M<1$ satisfying
\begin{equation}\label{assumptions}
\forall\theta\in\Theta,~\forall x\in\Z^d,~~|K^\theta_0(x)|\leq\frac{A}{1+\|x\|^{d+\tau}}~\mbox{and}~0<\hat K^\theta_0(x)\leq M.
\end{equation}
Let $X$ be the realization of a DPP on $\Z^d$ with kernel $K^{\theta^*}$, $\theta^*\in\Theta$. Then, for all $\theta\in\Theta$,
$$\sup_{\theta\in\Theta}\frac{1}{|W_n|}\big|\logdet(L_0^\theta[X\cap W_n])-\logdet(L^\theta_{[W_n]}[X\cap W_n])\big|\cvps 0.$$
\end{prop}

The only restrictive assumptions in Proposition \ref{approxLLCV} is the need for $K^\theta_0$ to decay faster than $\|x\|^{-d}$ and the fact that $\hat K^\theta_0$ never vanishes. In the usual setting where the kernels are parametric covariance functions (see Propositon \ref{intro:Fourier}), these assumptions are generally satisfied. That includes the Gaussian, Cauchy and Whittle-Matern kernels. The only exception amongst standard kernels is  the Bessel-type kernel, that will be examined by simulations in Section~\ref{sec bessel}.  Based on Propositions~\ref{DeterministicCV} and \ref{approxLLCV} and noticing that the assumptions of Proposition~\ref{approxLLCV} imply the assumptions of Proposition~\ref{DeterministicCV}, we thus obtain the consistency of the likelihood approximation  \eqref{approxLL}  when $\X=\Z^d$.

\begin{corr}\label{cor discrete} Let $\{K^\theta,\theta\in\Theta\}$ be a family of translation-invariant DPP kernels on $\Z^d$ satisfying the assumptions of Proposition~\ref{approxLLCV}, then $\sup_{\theta\in\Theta}|\tilde l(\theta|X)-l(\theta|X)|\cvps 0$ for all $\theta\in\Theta$. 
\end{corr}

Getting the same result for DPPs on $\R^d$ is still an open problem. However the next proposition shows that a DPP on $\R^d$ can be approximated by a discrete DPP on an arbitrarily small regular grid of $\R^d$, for which Corollary~\ref{cor discrete} applies. Note that the assumptions on $\hat K_0$ below are satisfied for all standard parametric families, see Section~\ref{sec:ParFam}.
\begin{prop} \label{GridApprox2}
Let $X$ be a stationary DPP on $\R^d$ with kernel $K(x,y)=K_0(y-x)$, where $K_0$ is a square integrable function such that $\hat K_0$ takes values in $[0,1[$  and
$$\forall x\in\R^d,~~0\leq \hat K_0(x)\leq \frac{A}{1+\|x\|^{d+\tau}}$$
for some constant $A,\tau>0$. For all $\epsilon > 0$, define $X_\epsilon$ as the DPP on  $\Z^d$ with kernel $K_\epsilon(x,y)\defeq\epsilon^dK_0(\epsilon(y-x))$. Then, $X_\epsilon$ is well-defined for small enough $\epsilon$ and the distribution of $\epsilon X_\epsilon$, the DPP $X_\epsilon$ rescaled by a factor $\epsilon$, weakly converges to the distribution of $X$ when $\epsilon$ tends to $0$.
\end{prop}

In the end, Corollary~\ref{cor discrete} tells us that the asymptotic approximation of the log-likelihood \eqref{approxLL} is theoretically sounded for most classical parametric families of stationary DPPs on $\Z^d$ and, as a consequence of Proposition \ref{GridApprox2}, also theoretically sounded for any discrete approximation of continuous DPPs on an arbitrarily small regular grid of $\R^d$. 

\section{Application to standard parametric families}\label{sec: appli}

\subsection{Classical parametric families of stationary DPPs} \label{sec:ParFam}
A classical way of generating parametric families of stationary DPPs is the following result.
\begin{prop} \label{intro:Fourier}
Let $K_0:\X\mapsto\R$ be a bounded square integrable symmetric function on $\R^d$ such that its Fourier transform $\hat K_0$ takes values in $[0,1]$. Then, the function $K(x,y)\defeq K_0(y-x)$ is a DPP kernel on $(\X,\nu)$.
\end{prop}
This proposition is proved in \cite{Lavancier} in the case $\X=\R^d$. 
Since symmetric functions $K_0$ with non negative Fourier transform are covariance functions, this result implies that we can consider as many parametric families of DPPs as there are parametric families of covariance functions. The assumption that $\hat K_0\leq 1$ simply adds a bound on the parameters of the family. Various examples are presented and  studied in \cite{Biscio, Lavancier}. We provide in Table~\ref{parametric models} some examples in $\R^d$. Note that for simplification, we call in this table Bessel kernel the particular case of the Bessel kernel in \cite{Biscio} where the shape parameter is $\sigma=0$, and Cauchy kernel the particular case in \cite{Lavancier} where the shape parameter is $1/2$. If the shape parameter is different for these models, then closed formulas are available for $K_0$ and $\hat K_0$, but not for $L_0$ (see the next section and Table~\ref{L0 models}). 

\begin{table}
\begin{equation*}
\begin{array}{|l|c|c|c|}
\hline
\rule[0pt]{0pt}{15pt}  & K_0(x) & \hat K_0(x) & \rho_{\max}\\\hline
\rule[-10pt]{0pt}{30pt} \text{Gauss} & \rho\exp\left(-\frac{\|x\|^2}{\alpha^2}\right) &  \rho (\sqrt \pi \alpha)^d \exp(-\|\pi\alpha x\|^2) & (\sqrt{\pi}\alpha)^{-d}\\ \hline
\rule[-15pt]{0pt}{35pt}  \text{Bessel} &  \rho 2^{d/2}\Gamma(d/2+1)\frac{J_{d/2}(\sqrt{2d}\|y-x\|/\alpha)}{(\sqrt{2d}\|y-x\|/\alpha)^{d/2}} & \frac{\rho}{\rho_{\max}}\cara{\|x\|\leq\sqrt{d/(2\pi^2\alpha^2)}} & \frac{d^{d/2}}{(2\pi)^{d/2}\alpha^d\Gamma(d/2+1)} \\ \hline
\rule[-10pt]{0pt}{31pt}  \text{Cauchy} & \rho \left(1+\left\|\frac{x}{\alpha}\right\|^2\right)^{-\frac{d+1}{2}} & \frac{\rho(\sqrt\pi\alpha)^d\sqrt{\pi}}{\Gamma((d+1)/2)}e^{-\|2\pi\alpha x\|} & \frac{\Gamma((d+1)/2)}{\pi^{(d+1)/2}\alpha^d} \\\hline
\rule[-15pt]{0pt}{35pt} \text{WM} & \rho\frac{2^{1-\sigma}}{\Gamma(\sigma)}\left\|\frac{x}{\alpha}\right\|^\sigma K_\sigma\left(\left\|\frac{x}{\alpha}\right\|\right) & \rho\frac{\Gamma(\sigma+d/2)}{\Gamma(\sigma)} \frac{(2\sqrt\pi \alpha)^d}{(1+\|2\pi\alpha x\|^2)^{\sigma+d/2}} & \frac{\Gamma(\sigma)}{\Gamma(\sigma+d/2)(2\sqrt{\pi}\alpha)^d} \\\hline
\end{array}
\end{equation*}
\caption{Examples of parametric kernels $K_0$ on $\R^d$, along with their Fourier transform $\hat K_0$. For each family, the intensity is $\rho$ and the range parameter is $\alpha$. The existence condition $\hat\K_0\leq 1$ is equivalent to $\rho\leq\rho_{\max}$ where $\rho_{\max}$ is given in the last column. The Whittle-Mat\'ern model (WM) also contains a shape parameter $\sigma>0$. Here $J_{d/2}$ denotes the Bessel function of the first kind and $K_\sigma$ the modified Bessel function of  the second kind.}
\label{parametric models}
\end{table}

\subsection{Expressions of \texorpdfstring{$L_0$}{L0}}\label{sec:L0}

When computing the approximate log-likelihood $\tilde l(\theta|X)$ in \eqref{approxLL} or its edge-corrected version \eqref{approxLLtorus}, one has to compute $L_0(y-x)$ for each pair of points $(x,y)\in (X\cap W)^2$. It is thus important to find faster ways to compute values of $L_0$ than the $d$-dimensional integral \eqref{app1}. 
An important example arises when $K_0$ is a radial function, denoted by $K_{\textrm{\rad}}$. In this case, the corresponding DPP is isotropic and $L_0$ is also a radial function, denoted by $L_{\textrm{\rad}}$. The Fourier transform can then be expressed by a Hankel transform which gives
$$\hat K_{\textrm{rad}}(r)=\frac{2\pi}{r^{d/2-1}}\int_0^{\infty}s^{d/2}K_{\textrm{rad}}(s)J_{d/2-1}(2\pi sr)\der s$$
and
$$L_{\textrm{rad}}(r)=\frac{2\pi}{r^{d/2-1}}\int_0^{\infty}s^{d/2}\frac{\hat K_{\textrm{rad}}(s)}{1-\hat K_{\textrm{rad}}(s)}J_{d/2-1}(2\pi sr)\der s.$$
The expression of $L_0$ therefore simplifies into a unidimensional integral.

Moreover, we may exploit the relation $\hat L_0=\hat K_0/(1-\hat K_0)=\sum_{n\geq 1}(\hat K_0)^n$ and try to compute the inverse Fourier transform to express $L_0$ as a series with exponentially decreasing coefficients (see the discussion in Section~\ref{sec:periodic_correction}) or even get  an analytic expression.  This strategy leads to closed-form formulas of $L_0$ for the classical parametric families displayed in Table~\ref{parametric models}. The results, obtained after straightforward calculus, are given in Table~\ref{L0 models}.

\begin{table}
\begin{equation*}
\begin{array}{|l|*1{>{\displaystyle}c|}}
\hline
\rule[0pt]{0pt}{15pt}  & L_0(x) \\\hline
\rule[-20pt]{0pt}{50pt} \text{Gauss} & \sum_{n\geq 1}\rho^n\frac{(\sqrt{\pi}\alpha)^{d(n-1)}}{n^{d/2}}\exp\left(-\frac{\|x\|^2}{n\alpha^2}\right) \\ \hline
\rule[-25pt]{0pt}{50pt} \text{Bessel} & \frac{\rho 2^{d/2}\Gamma(d/2+1)}{1-\rho\frac{(2\pi)^{d/2}\alpha^d\Gamma(d/2+1)}{d^{d/2}}}\frac{J_{d/2}(\sqrt{2d}\|x\|/\alpha)}{(\sqrt{2d}\|x\|/\alpha)^{d/2}} \\ \hline
\rule[-20pt]{0pt}{50pt} \text{Cauchy} & \sum_{n\geq 1}\frac{\rho^n}{n^d}\left(\frac{\pi^{(d+1)/2}\alpha^d}{\Gamma((d+1)/2)}\right)^{n-1}\left(1+\left\|\frac{x}{n\alpha}\right\|^2\right)^{-(d+1)/2} \\\hline
\rule[-20pt]{0pt}{50pt}\text{WM} & \sum_{n\geq 1}\frac{\rho^n(\sqrt{\pi}\alpha)^{d(n-1)}\Gamma(\sigma+d/2)^n}{2^{n\sigma-1-(n-1)d/2}\Gamma(\sigma)^n\Gamma(n\sigma+nd/2)}\left\|\frac{x}{\alpha}\right\|^{n\sigma+(n-1)d/2}\hspace{-1cm}K_{n\sigma+(n-1)d/2}\left(\left\|\frac{x}{\alpha}\right\|\right) \\\hline
\end{array}
\end{equation*}
\caption{Expression of $L_0$ defined in \eqref{app1} for the parametric kernels given in Table~\ref{parametric models}. }
\label{L0 models}
\end{table}

\subsection{Estimation of the intensity by MLE}\label{sec:Estrho}

Assume that the parametric DPP kernel reads for some parameters $\rho$ and $\theta$
\begin{equation}\label{Estrho}
K^{\rho,\theta}(x,y)=\rho\tilde K^{\theta}(x,y)
\end{equation}
where $\tilde K^{\theta}(x,x)=1$ for all $x$. The parameter $\rho$ corresponds here to the intensity of the DPP and $\theta$ to the other parameters of the model. This is the setting of all standard parametric models, including those presented in Table~\ref{parametric models}. 

When jointly estimating $(\rho,\theta)$ from a realization of the DPP $X$ on $W$ by the approximate MLE, simulations usually show that the estimate of $\rho$ appears to be very close to $N(W)/|W|$. One explanation given in \cite{Lavancier} is that, by doing a first order convolution approximation in \eqref{convol1} and \eqref{convol2}, we get
$$l(\rho,\theta | X)\approx 1-\rho+\log(\rho)\frac{N(W)}{|W|}+\frac1 {|W|}\logdet(\tilde K_{W}^\theta[X\cap W])$$
and the maximum point of this approximation is $\hat\rho=N(W)/|W|$. We even show in Proposition \ref{rhocool} that, in the case of Bessel type DPP kernels with parameters $(\rho,\alpha)$ as presented in Table~\ref{parametric models}, $\hat\rho=N(W)/|W|$ is always the maximum point of $\rho\mapsto\tilde l(\rho,\alpha|X)$ for any $\alpha$. This result suggests that, instead of jointly estimating $\rho$ and $\theta$, it is more computationally efficient to directly estimate $\rho$ by $\hat \rho=N(W)/|W|$ and then $\theta$ by an argument of the maximum of $\theta\mapsto\tilde l(\hat\rho,\theta|X)$.

\subsection{Estimation of the MLE standard errors}\label{sec:sd}

For most statistical models, the MLE is expected to have an asymptotic variance equal to the inverse Fisher information matrix. Even if this property is not theoretically proved for DPPs' models, it is a natural conjecture to make. To estimate this  variance, it is common (and even advocated in \cite{Efron_fisher}) to use the observed information,  which is the matrix with entries $-|W|\partial_{\theta_i}\partial_{\theta_j}  l(\theta|X)$, whose expectation defines the genuine Fisher information matrix. 

Since our approximation  \eqref{approxLL} is generally smooth in the parameters (see below), we may consider the following  approximation of the observed information matrix:
$$\tilde I(\theta)\defeq -|W|\left(\partial_{\theta_i}\partial_{\theta_j} \tilde l(\theta|X)\right)_{1\leq i,j\leq p}.$$
If $\hat\theta$ is the approximated MLE based on  \eqref{approxLL}, we can thus estimate its variance  by  
$\tilde I(\hat\theta)^{-1}$.

This estimation is possible whenever $\theta\mapsto K^\theta_0(x)$ and $\theta\mapsto L^\theta_0(x)$ are twice differentiable on $\Theta\subset\R^p$ for all $x\in\X$. Then so is $\theta\mapsto \tilde l(\theta|X)$ and we obtain 
\begin{multline}\label{eq:2nd_derivative_LL}
\partial_{\theta_i}\partial_{\theta_j} \tilde l(\theta|X) = \int_{\X^*}\frac{-(\partial_{\theta_i}\partial_{\theta_j} \hat{K}^\theta_0(x))(1-\hat K^\theta_0(x))-\partial_{\theta_i} \hat{K}^\theta_0(x)\partial_{\theta_j}\hat{K}^\theta_0(x)}{(1-\hat K^\theta_0(x))^2}\der x  \\
+\frac{1}{|W|}\tr\left((\partial_{\theta_i}\partial_{\theta_j} L^\theta_0) (L^\theta_0)^{-1}-(\partial_{\theta_i} L^\theta_0)  (L^\theta_0)^{-1}(\partial_{\theta_j} L^\theta_0) (L^\theta_0)^{-1}\right),
\end{multline}
where we have written $L^\theta_0$ for $L^\theta_0[X\cap W]$. 
Each derivative in this expression can easily be deduced from Tables~\ref{parametric models} and \ref{L0 models} for the parametric models discussed before.

Note that such variance estimation is not possible for the Fourier series approximation \eqref{eq:Fourier_approx} because this approximation is not differentiable in general, as illustrated in Figure~\ref{ComparingSmoothness} for the scale parameter of the Gaussian kernel, so that the observed information is not well defined in this case. 
Moreover, concerning the alternative minimum contrast estimators of a parametric DPP model considered in \cite{Lavancier,biscio:lavancier:17}, no tractable formulas are available for their asymptotic variance. For these estimation methods, the only way to approximate the associated standard errors is parametric bootstrap, a very time consuming procedure.

\section{Simulation study} \label{sec:Simu}

In this section we perform a simulation study to investigate the performance of our approximate MLE, with and without edge effect correction, and compare it to minimum contrast estimators (MCE for short) based on Ripley's $K$ function and on the pair correlation function (pcf for short), both being common second-order moment estimators used in spatial statistics. We refer to \cite{biscio:lavancier:17} for more detailed information on these MCEs applied to DPPs. At the exception of the special case of Bessel-type DPPs considered in Section~\ref{sec bessel}, we chose not to compare our estimators to the Fourier approximation \eqref{eq:Fourier_approx} of \cite{Lavancier} since, as explained in Section~\ref{sec:periodic_correction}, this estimator yields nearly the same results as our corrected MLE, which we confirmed in our testings, with the notable difference of the Fourier approximation being about ten times longer to compute in our examples.

\subsection{Whittle-Mat\'ern, Cauchy and Gaussian-type DPPs}\label{sec:Simu_Pas_Bessel}

We consider in this section the parametric models in Table~\ref{parametric models} that are covered by our theoretical assumptions in Section~\ref{section theory}, that are the Whittle-Mat\'ern, Cauchy and Gaussian-type DPPs. From this perspective, these are favourable models  for our likelihood approximation approach. All these models are of the form \eqref{Estrho}, then following Section \ref{sec:Estrho},  we estimate $\rho$ by $\hat\rho=N(W)/|W|$ for all methods, and  the performances are evaluated on the estimation of $\alpha$ only. Note that for  the Whittle-Mat\'ern model, we do not consider   the estimation of the shape parameter $\sigma$, which was assumed to be known. The joint estimation of $(\alpha,\sigma)$ for this model is known to be a poorly identifiable problem and it is customary to choose the best $\sigma$ from a small finite grid by profile likelihood  (see \cite{Lavancier}).  
For the estimation of $\alpha$, we have performed the same kind of simulations for the three models in $\R^2$. The results and conclusions are similar. In the following we only present the details for the Gaussian-type DPP but the code used to produce all results is provided as supplementary material and at \url{https://github.com/APoinas/MLEDPP}.

\begin{figure}
\centering
\begin{tabular}{ccc}
   \includegraphics[width=0.3\textwidth]{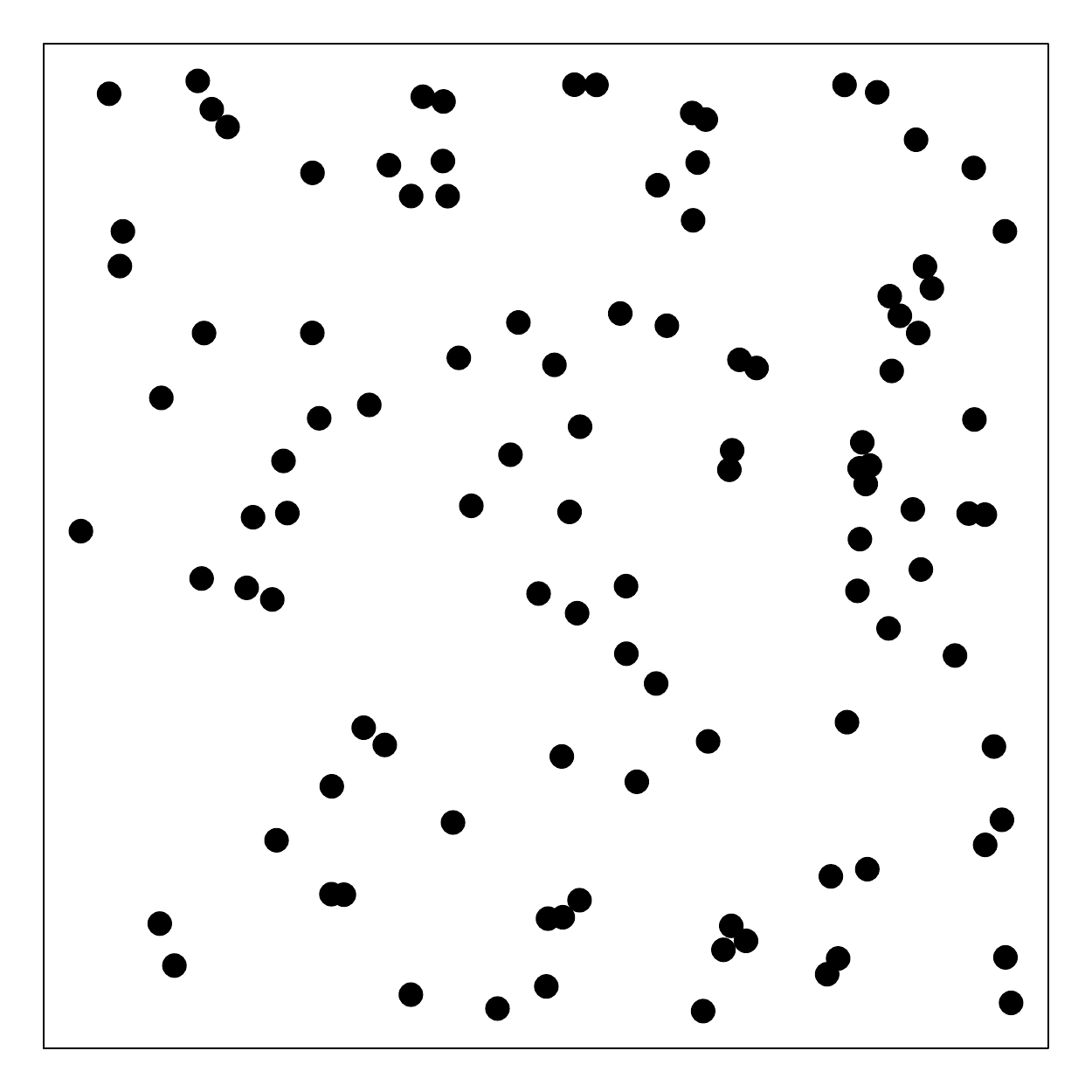} &
      \includegraphics[width=0.3\textwidth]{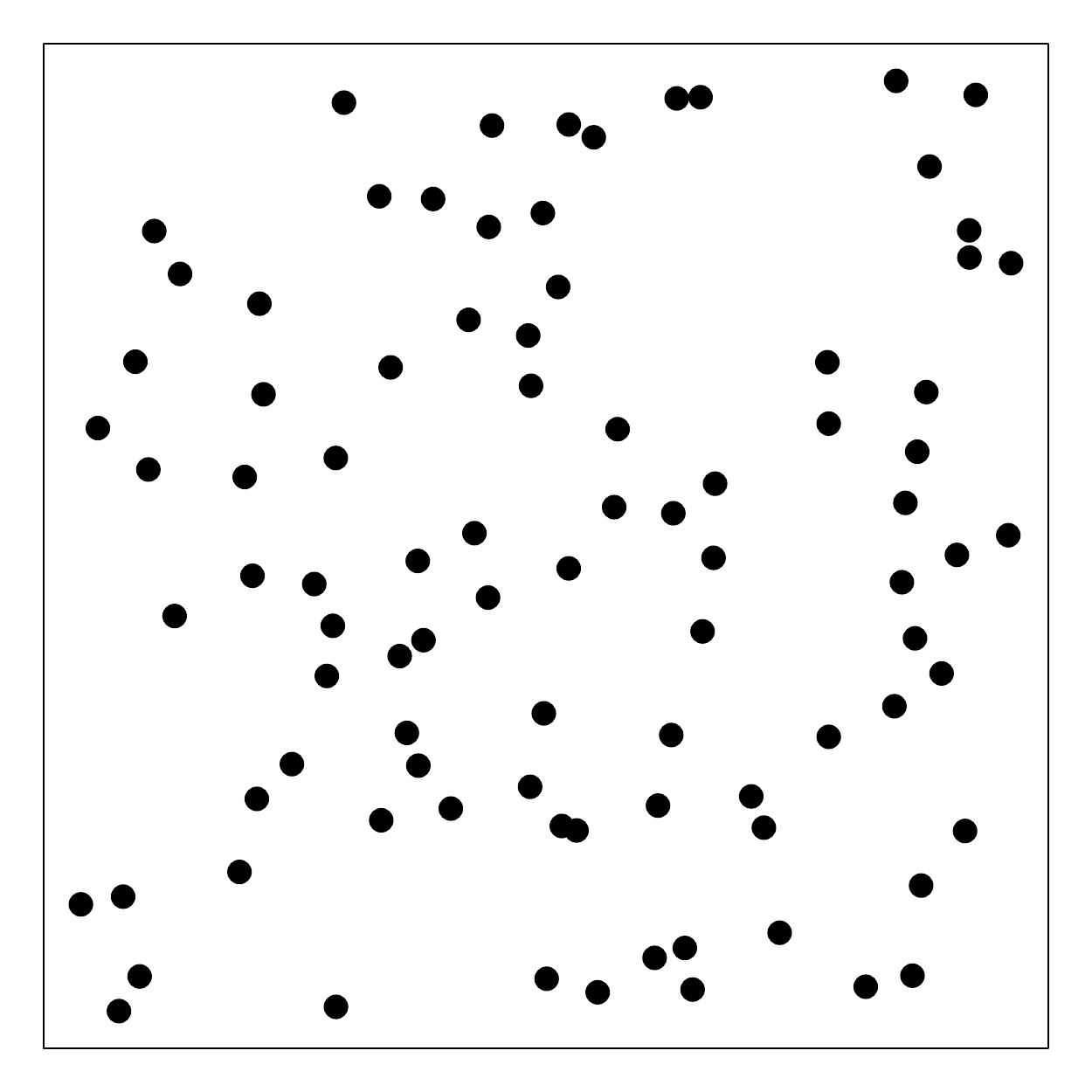} &
   \includegraphics[width=0.3\textwidth]{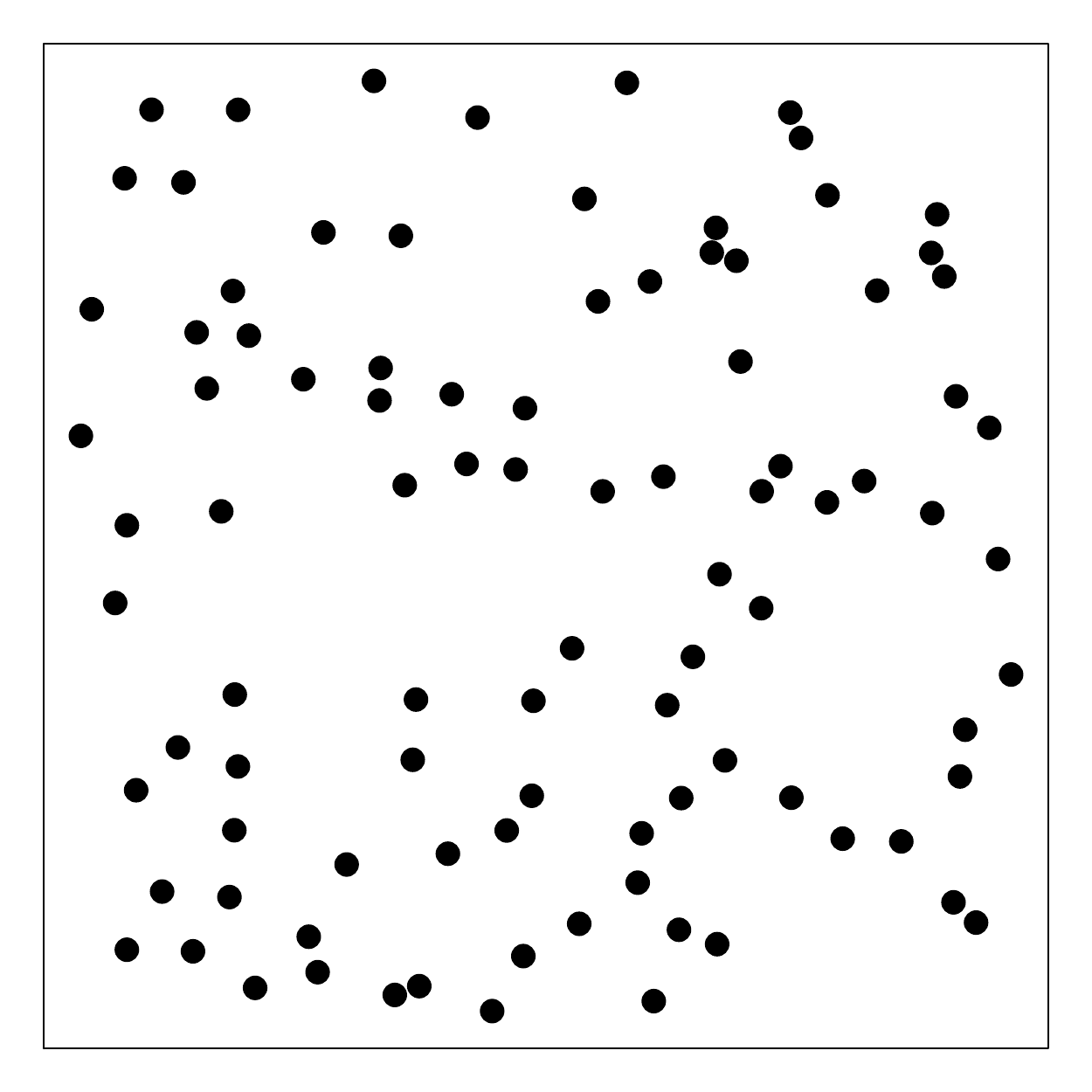} 
\end{tabular}
	   \caption{\label{ExampleReal} Examples of realizations of Gaussian-type DPPs on $[0,1]^2$ with parameters $\rho^*=100$ and $\alpha^*=0.01,0.03,0.05$ corresponding to three different degrees of repulsiveness.}
\end{figure}

We consider realizations of the Gaussian-type DPP with true parameters $\rho^*=100$ and $\alpha^*\in\{0.01,0.03,0.05\}$, when the observation window $W$ is either $[0,1]^2$, $[0,2]^2$ or $[0,3]^2$. When $\rho^*=100$, $\alpha$ can take values in $]0,(10\sqrt{\pi})^{-1}\approx 0.056[$ since the process exists if and only if  $\pi\rho\alpha^2\leq 1$. Therefore, $\alpha^*=0.01$ corresponds to a weakly repulsive point process, close to a Poisson point process, while $\alpha^*=0.03$ corresponds to a mildly repulsive DPP and $\alpha^*=0.05$ corresponds to a strongly repulsive DPP. Examples of realizations are shown in Figure \ref{ExampleReal}. We estimate $\alpha^*$ by the approximate MLE defined in \eqref{approxLL} and compare it to its edge-corrected version defined in \eqref{approxLLtorus} as well as MCEs based on the pcf or Ripley's $K$ function. As mentioned before, $\rho$ is replaced by $\hat\rho=N(W)/|W|$ in \eqref{approxLL} and \eqref{approxLLtorus}. Moreover we truncate the series defining $L_0$ (see Table~\ref{L0 models}) to the minimal value of $n\leq 50$ such that all remainder terms in the series become less than $10^{-4}$ times the first term. This choice leads to $n\leq 10$ for most values of $\alpha$ and to $n=50$ only for $\alpha > 0.9 \alpha_{\max}$, where $\alpha_{\max}=1/\sqrt{\pi\rho}$. All realizations have been generated in R \cite{R} using the \textit{spatstat} \cite{baddeley:rubak:turner:15} package and both MCEs were computed by the function \textit{dppm} of the same package. The tuning parameters for these MCEs were $r_{\min}=0.01$, $r_{\max}$ being one quarter of the side length of the window and $q=0.5$ as recommended in \cite{Diggle}. Boxplots of the difference between the four considered estimators and the true value $\alpha^*$ for $500$ runs in all different cases are displayed in Figure \ref{Boxplots} and the  corresponding mean square errors are given in Table \ref{MSEtable}.

\begin{figure}
\centering
{\small
\begin{tabular}{b{1.2cm}m{0.5cm}*4{>{\centering}p{0.17\textwidth}}}
 \raisebox{2cm}{$\alpha^*=0.01$} & \multicolumn{5}{l}{ \includegraphics[width=0.85\textwidth,height=0.2\textheight]{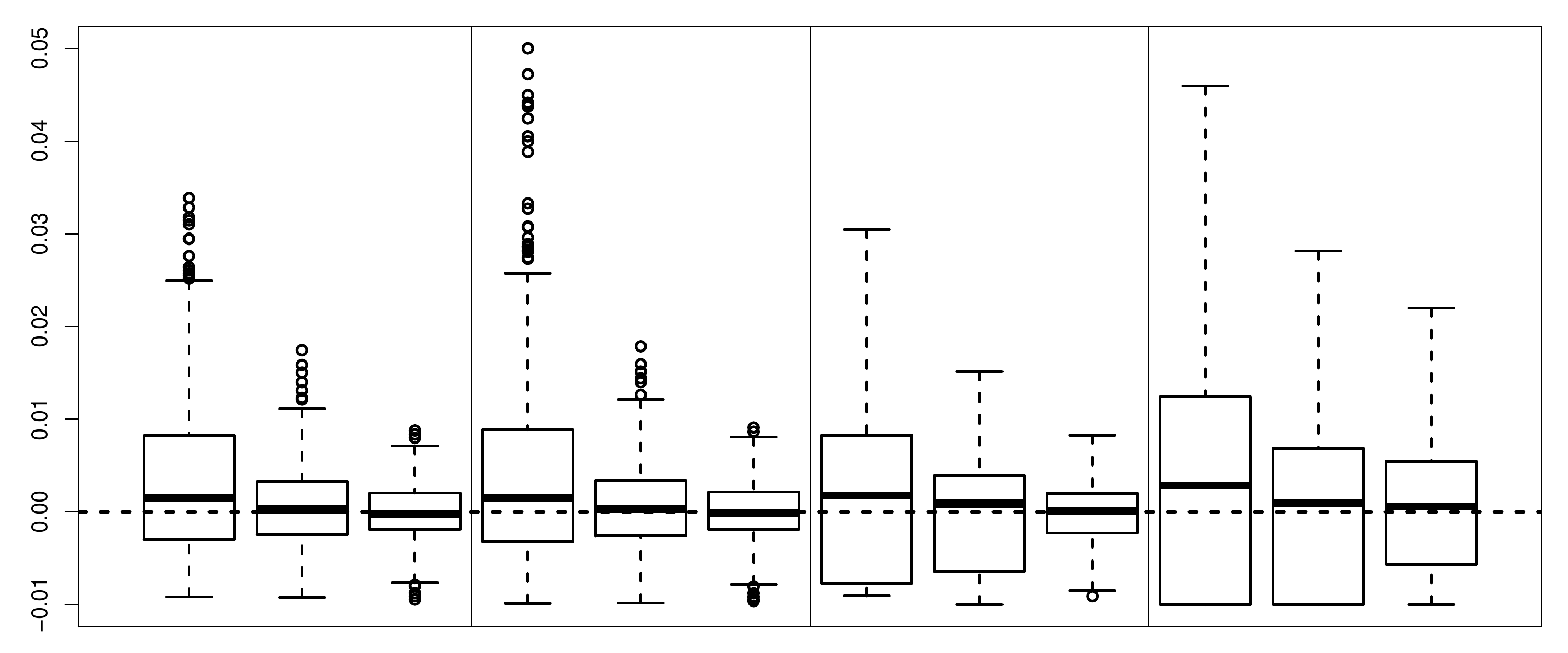}
}\\
 \raisebox{2cm}{$\alpha^*=0.03$} & \multicolumn{5}{l}{ \includegraphics[width=0.85\textwidth,height=0.2\textheight]{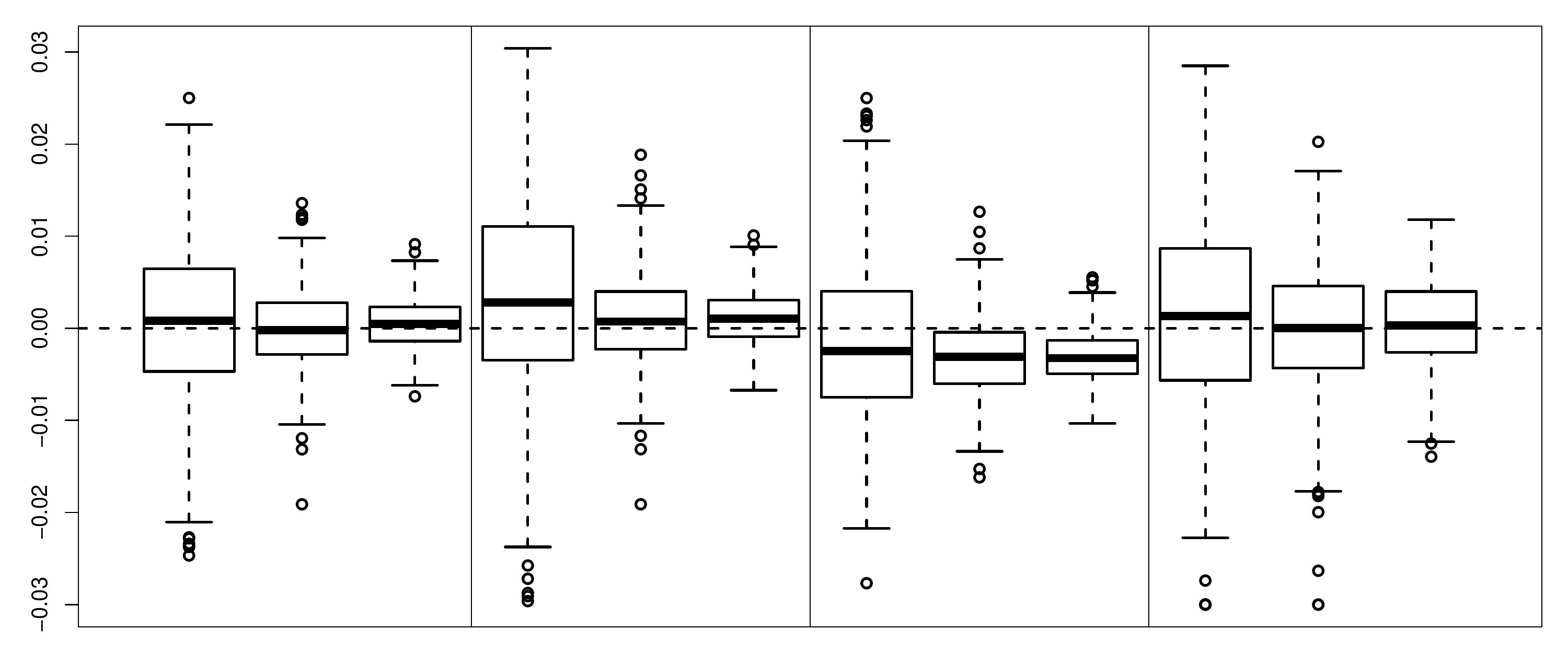}
}\\
 \raisebox{2cm}{$\alpha^*=0.05$} & \multicolumn{5}{l}{ \includegraphics[width=0.85\textwidth,height=0.2\textheight]{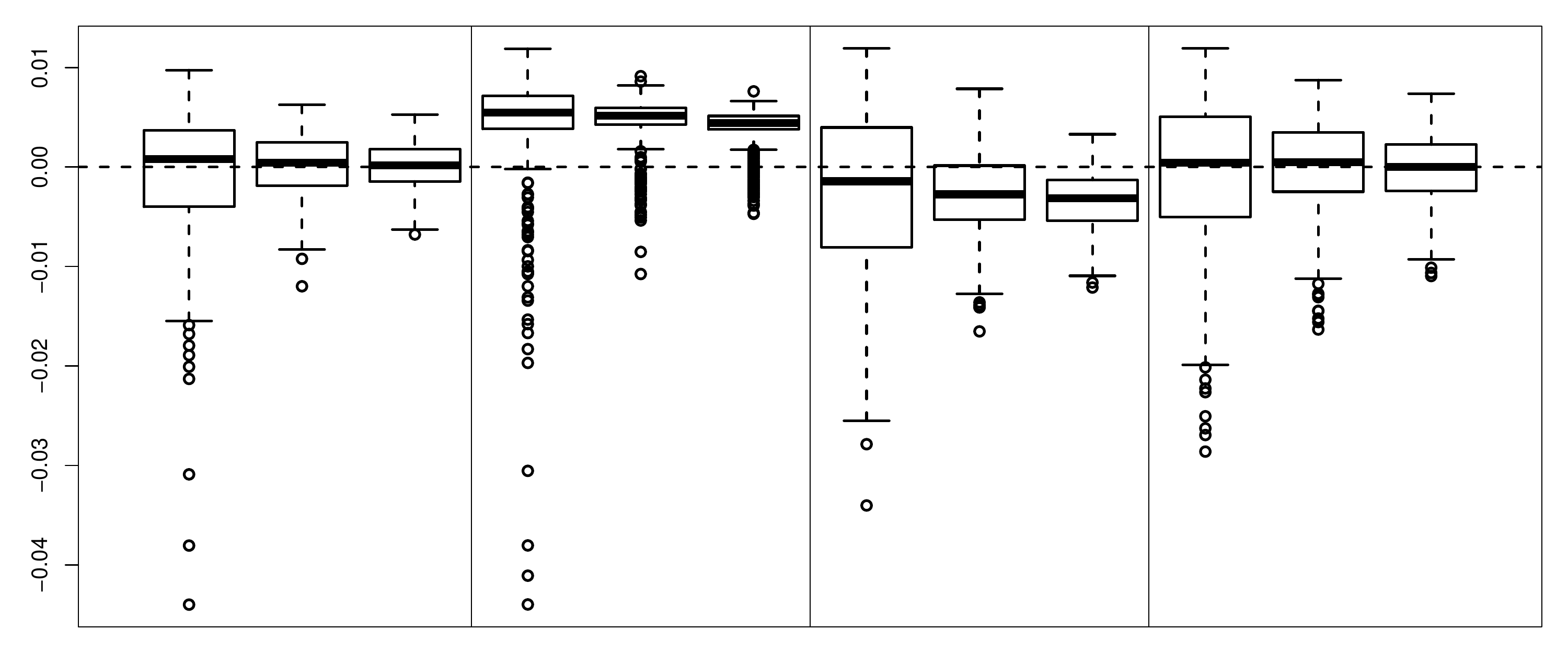}
}\\
& & $\tilde l^{~\T}(\theta|X)$ & $\tilde l(\theta|X)$ & pcf & Ripley 
\end{tabular}
}
 \caption{\label{Boxplots} Boxplots of $\hat\alpha-\alpha^*$ generated from $500$ simulations of Gaussian-type DPPs with true parameters $\rho^*=100$ and, from top to bottom, $\alpha^*=0.01$, $0.03$ and $0.05$. Each row shows the behaviour of  the following 4 estimators when the simulation window is, from left to right in each box, $W=[0,1]^2$, $[0,2]^2$ and $[0,3]^2$:  the approximate MLE with edge-corrections based on $\tilde l^{~\T}(\hat\rho,\alpha|X)$, the approximate MLE based on $\tilde l(\hat\rho,\alpha|X)$, the MCE based on the pair correlation function and the MCE based on the Ripley's $K$ function.}
	 \end{figure}

\begin{table}[h]
\centering
{\renewcommand{\arraystretch}{1.25}
\begin{tabular}{|c|c|c|c|c|c|c|c|c|c|}
\hline
Window & \multicolumn{3}{c|}{$[0,1]^2$} & \multicolumn{3}{c|}{$[0,2]^2$} & \multicolumn{3}{c|}{$[0,3]^2$} \\
\hline
$\alpha^*$ & 0.01 & 0.03 & 0.05 & 0.01 & 0.03 & 0.05 & 0.01 & 0.03 & 0.05 \\
\hline
MLE based on $\tilde l^{~\T}$ & \textbf{0.83} & 0.81 & \textbf{0.41} & \textbf{0.21} & \textbf{0.18} & \textbf{0.088} & \textbf{0.090} & \textbf{0.079} & \textbf{0.051}\\
MLE based on $\tilde l$ & 1.25 & 1.75 & 0.54 & 0.24 & 0.23 & 0.28 & 0.095 & 0.10 & 0.20 \\
MCE (pcf) & 0.86 & \textbf{0.77} & 0.74 & 0.31 & 0.27 & 0.23 & 0.17 & 0.17 & 0.19\\
MCE ($K$) & 1.81 & 1.17 & 0.51 & 0.74 & 0.46 & 0.21 & 0.48 & 0.23 & 0.12\\
\hline
\end{tabular}}
\caption{{ Estimated mean square errors (x$10^4$) of $\hat\alpha$ for Gaussian-type DPPs on different windows and with different values of $\alpha$, each computed from $500$ simulations. 
}} \label{MSEtable} \end{table}

From these results, we remark that when $\alpha^*=0.01$ and $\alpha^*=0.03$, inference based on the approximate likelihood $\tilde l(\hat\rho,\alpha|X)$ outperforms moment based inference for windows bigger than $[0,2]^2$. This is expected from maximum likelihood based inference and shows that hundreds of points are enough for $\tilde l(\hat\rho,\alpha|X)$ to be a good  approximation of the true likelihood when the underlying DPP is not too repulsive. 
When $\alpha^*=0.05$, that is when the negative dependence of the DPP is very strong, then $\tilde l(\hat\rho,\alpha|X)$ suffers from edge effects and is heavily biased. In fact, as can be seen in Figure \ref{CorrectionPlots}, $\tilde l(\hat\rho,\alpha|X)$ is an increasing function of $\alpha$ in this case and the estimate is often the highest possible value for $\alpha$, which is $1/\sqrt{\pi\hat\rho}$.  The correction $\tilde l^{~\T}$ introduced in \eqref{approxLLtorus} gives more accurate values of the likelihood for high values of $\alpha$, as shown in Figure \ref{CorrectionPlots}. Finally this estimator outperforms the other ones in nearly every cases and especially the most repulsive ones. 

\begin{figure}
\centering
\begin{tabular}{ccc}
   \includegraphics[width=0.3\textwidth]{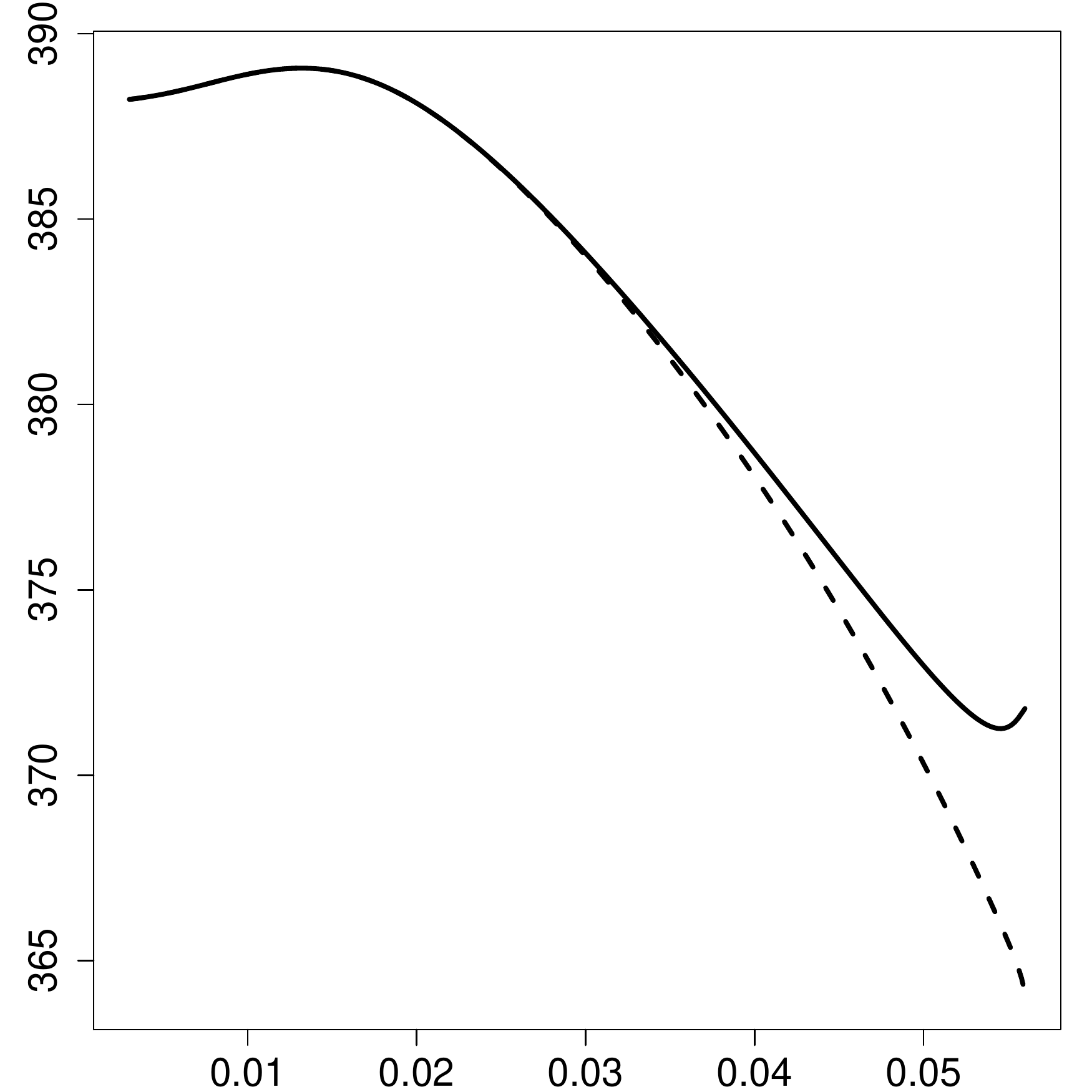} &
      \includegraphics[width=0.3\textwidth]{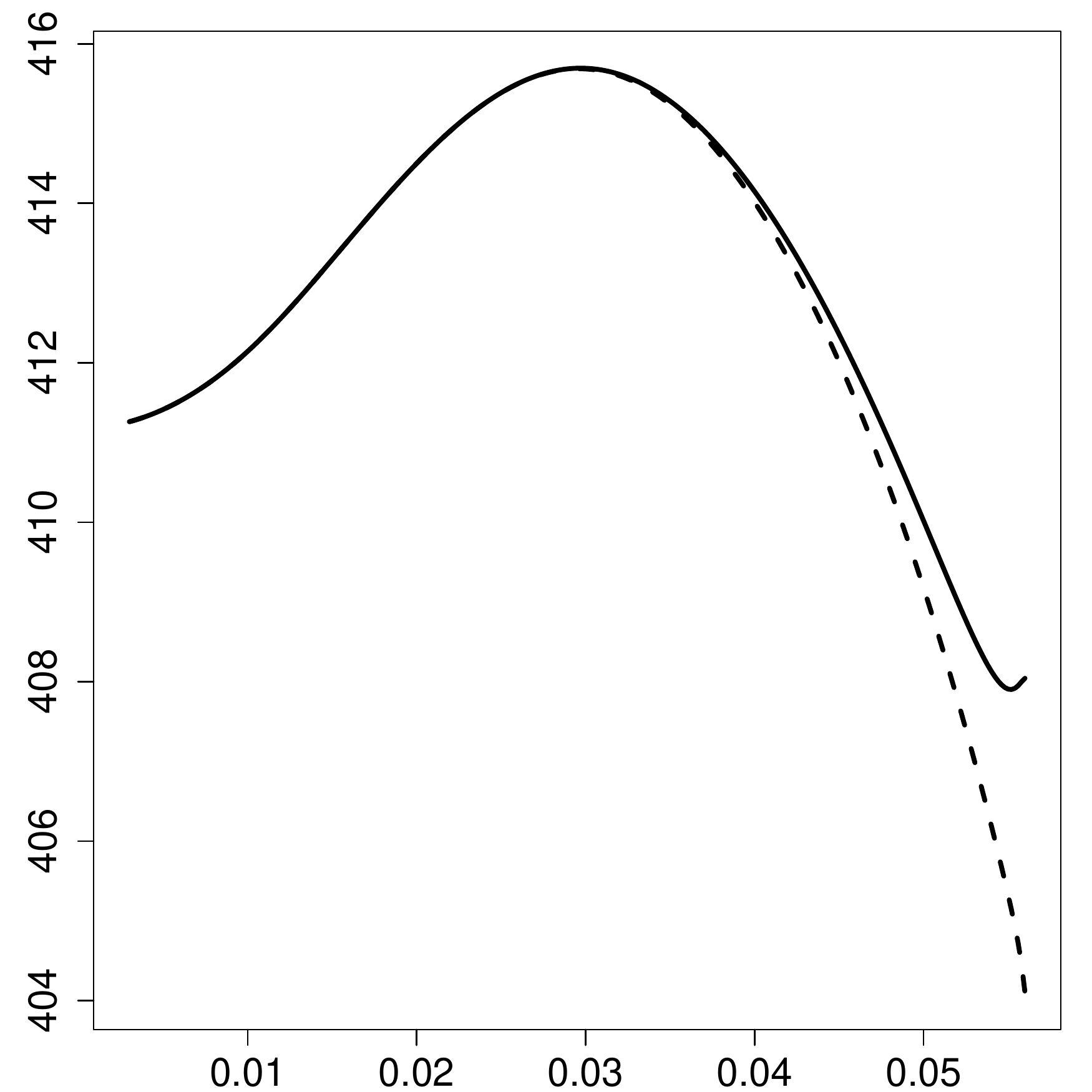} &
   \includegraphics[width=0.3\textwidth]{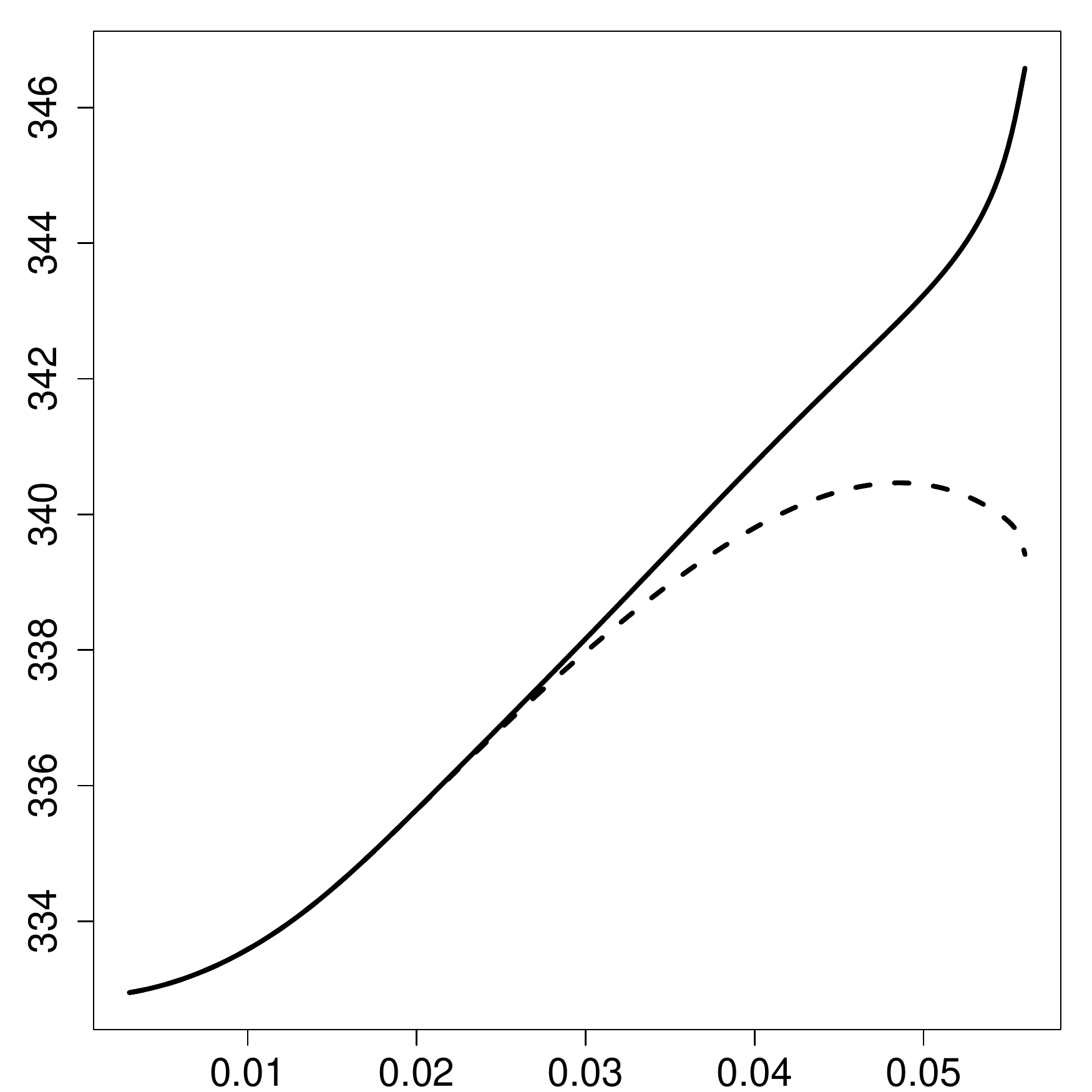} 
\end{tabular}
	  \caption{\label{CorrectionPlots} Comparison between $\tilde l(100,\alpha|X)$ (solid lines) and $\tilde l^{~\T}(100,\alpha|X)$ (dashed lines) with respect to $\alpha$ where $X$ is one realization of a DPP on $[0,1]^2$ with a Gaussian-type kernel  with true parameters $\rho^*=100$ and, from left to right, $\alpha^*=0.01$, $0.03$ and $0.05$.}
\end{figure}

Concerning the computation time, even if our MLE approximation is much faster than the Fourier approximation \eqref{eq:Fourier_approx}, it can be heavy due to the need to optimize a function defined as the log-determinant of an $n\times n$ matrix, where $n$ is the number of observed points. For comparison, each MCE took less than one second on a regular laptop in each case considered in Figure \ref{Boxplots}, while each computation of the approximate MLE took between 1 and 2 seconds when $W=[0,2]^2$ and about $10$ seconds when $W=[0,3]^2$.

\subsection{Performance for Bessel-type DPPs}\label{sec bessel}

In order to evaluate the possible limitations of our approach, we consider in this section the estimation of Bessel-type DPPs, see Table~\ref{parametric models}, whose kernels do not satisfy the theoretical assumptions in Section~\ref{section theory}. 
As in the previous section, we set $\rho^*=100$, $\alpha^*=0.01, 0.03, 0.05$, corresponding to weak, medium and strong repulsiveness, and the observation window is $[0,1]^2$, $[0,2]^2$ and $[0,3]^2$. The results on 500 runs in each situation are shown in Figure~\ref{Boxplots Bessel} and in Table~\ref{MSEtable Bessel}. They compare our edge-correction approximate MLE, the Fourier series approximation \eqref{eq:Fourier_approx}, the MCE based on the pair correlation function and the MCE based on the Ripley's $K$ function. The performances are globally in line with the observations made in the previous section, showing that the approximate MLE outperforms MCEs, especially when the observation windows is large enough. Note that we have added the Fourier series approximation for comparison, because contrary to the models considered in the previous section, its behavior slightly differs from our edge-correction approximation for  Bessel-type DPPs, as discussed in the following.

\begin{figure}
\centering
{\small
\begin{tabular}{b{1.2cm}m{0.5cm}*4{>{\centering}p{0.175\textwidth}}}
 \raisebox{2cm}{$\alpha^*=0.01$} & \multicolumn{5}{l}{ \includegraphics[width=0.85\textwidth,height=0.2\textheight]{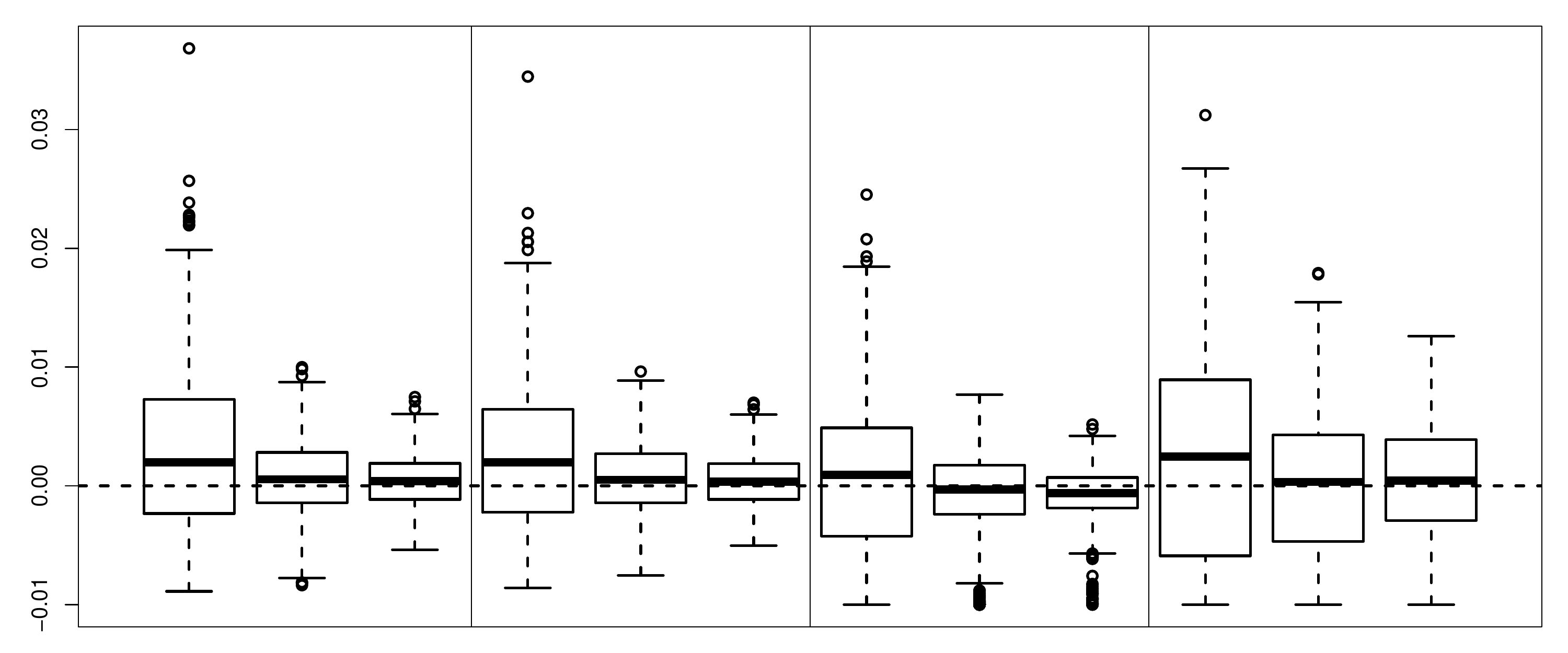}
}\\
 \raisebox{2cm}{$\alpha^*=0.03$} & \multicolumn{5}{l}{ \includegraphics[width=0.85\textwidth,height=0.2\textheight]{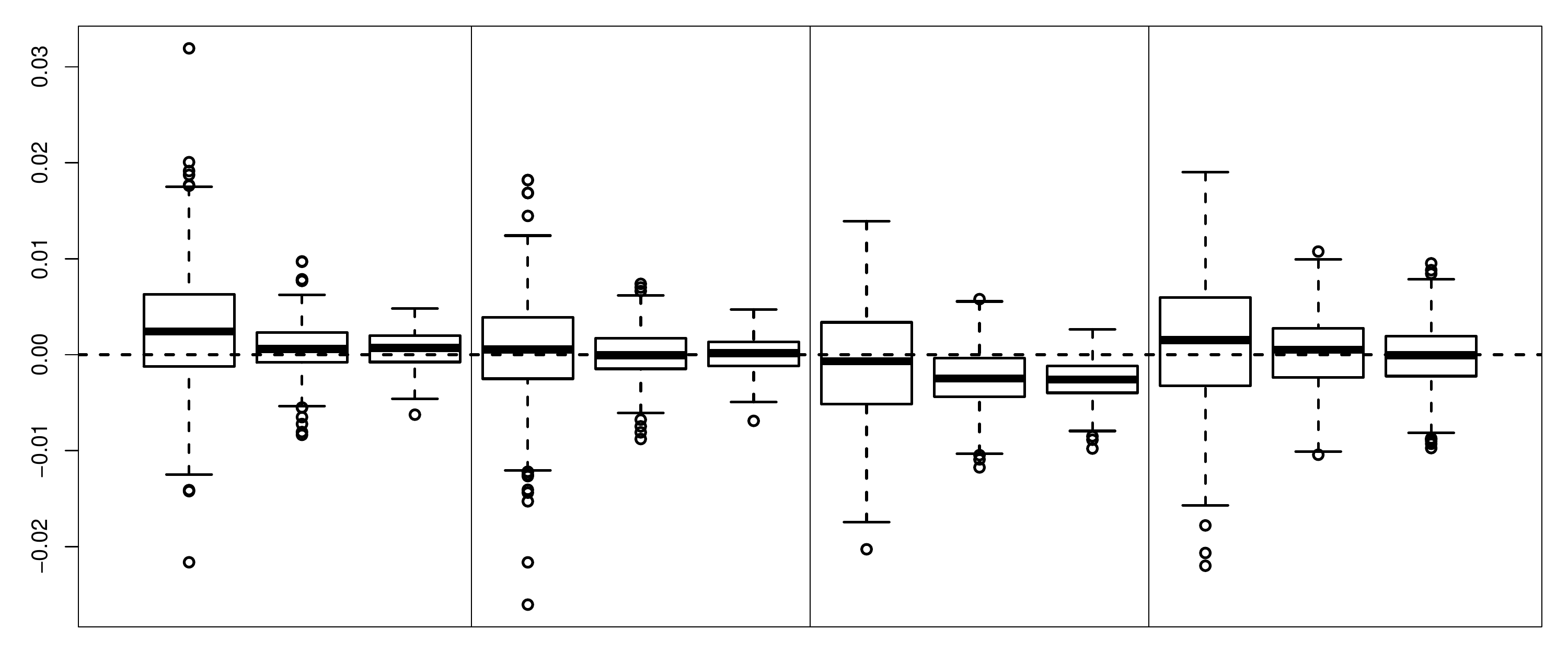}
}\\
 \raisebox{2cm}{$\alpha^*=0.05$} & \multicolumn{5}{l}{ \includegraphics[width=0.85\textwidth,height=0.2\textheight]{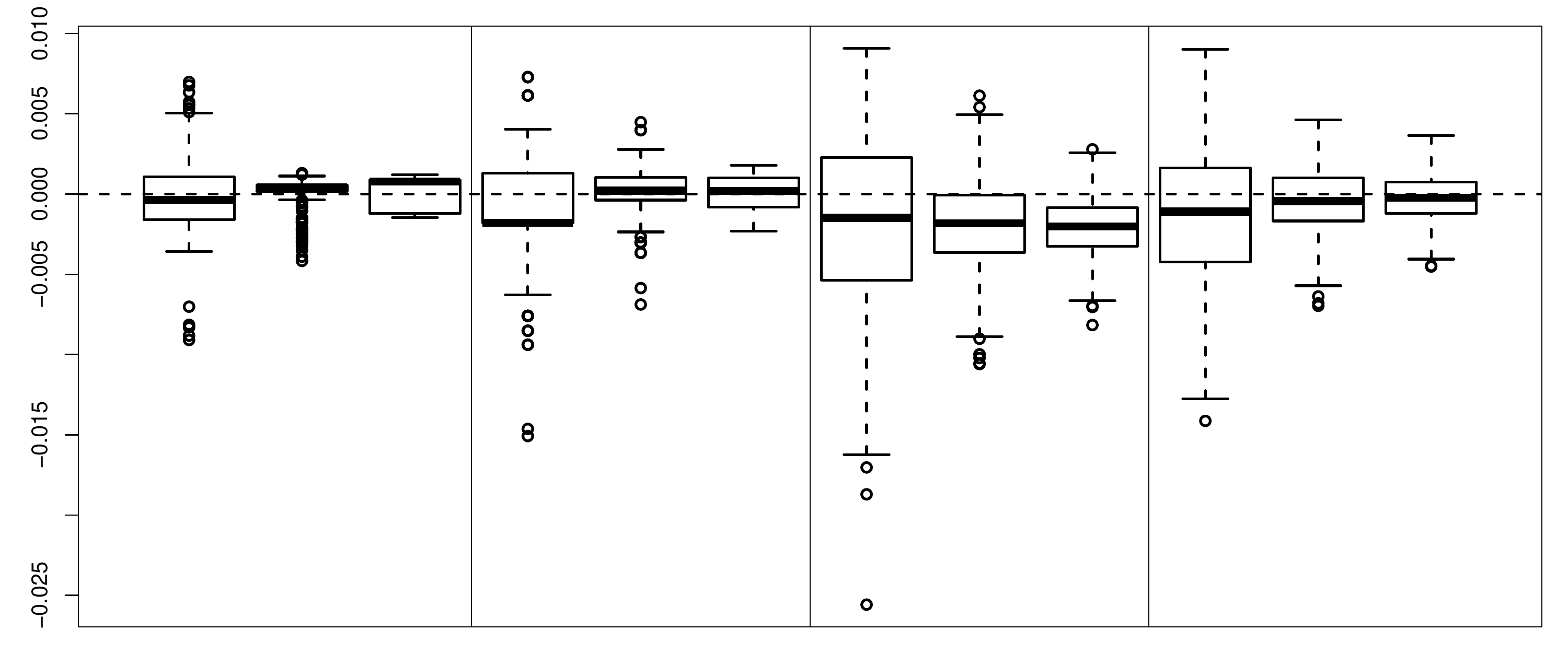}
}\\
& & $\tilde l^{~\T}(\theta|X)$ & Fourier approx. & pcf & Ripley
\end{tabular}
}
 \caption{\label{Boxplots Bessel} Boxplots of $\hat\alpha-\alpha^*$ generated from $500$ simulations of Bessel-type DPPs with true parameters $\rho^*=100$ and, from top to bottom, $\alpha^*=0.01$, $0.03$ and $0.05$. Each row shows the behaviour of the following 4 estimators when the simulation window is, from left to right in each box, $W=[0,1]^2$, $[0,2]^2$ and $[0,3]^2$:  the approximate MLE with edge-corrections based on $\tilde l^{~\T}(\hat\rho,\alpha|X)$, the Fourier series approximate MLE \eqref{eq:Fourier_approx}, the MCE based on the pair correlation function and the MCE based on the Ripley's $K$ function.}
\end{figure}

\begin{table}
\centering
{\renewcommand{\arraystretch}{1.25}
\begin{tabular}{|c|c|c|c|c|c|c|c|c|c|}
\hline
Window & \multicolumn{3}{c|}{$[0,1]^2$} & \multicolumn{3}{c|}{$[0,2]^2$} & \multicolumn{3}{c|}{$[0,3]^2$} \\
\hline
$\alpha^*$ & 0.01 & 0.03 & 0.05 & 0.01 & 0.03 & 0.05 & 0.01 & 0.03 & 0.05 \\
\hline
MLE based on $\tilde l^{~\T}$ & 0.56 & 0.49 & \textbf{0.04} & 0.12 & 0.08 & \textbf{0.01} & \textbf{0.05} & \textbf{0.03} & \textbf{0.01}\\
Fourier approx. MLE & \textbf{0.47} & \textbf{0.32} & 0.09 & \textbf{0.11} & \textbf{0.06} & 0.02 & \textbf{0.05} & \textbf{0.03} & \textbf{0.01} \\
MCE (pcf) &  0.50  & 0.39 & 0.33 & 0.21 & 0.14 & 0.11 & 0.10 & 0.11 & 0.07\\
MCE ($K$) & 0.95 & 0.46 & 0.19 & 0.41 & 0.15 & 0.04 & 0.27 & 0.10 & 0.02\\
\hline
\end{tabular}}
\caption{{ Estimated mean square errors (x$10^4$) of $\hat\alpha$ for Bessel-type DPPs on different windows and with different values of $\alpha$, each computed from $500$ simulations. 
}} \label{MSEtable Bessel} \end{table}

Despite the decent results of our approximation for Bessel-type DPPs, some issues appear with this model in the most repulsive case $\alpha^*=0.05$. As noticed in Section~\ref{sec:periodic_correction}, the determinant in \eqref{approxLLtorus} may be negative for high values of $\alpha$, making the computation of the approximate likelihood impossible. This problem is illustrated in the rightmost plot of Figure~\ref{Bessel issues}, that shows an example of an approximated likelihood function as in  \eqref{approxLLtorus} from one realization of a Bessel-type DPP on $W=[0,3]^2$ with $\rho^*=100$ and $\alpha^*=0.05$. The cross-type points on the right of this plot indicate the values of $\alpha$ where the determinant was negative. More generally, for the highest values of $\alpha$, the approximate likelihood is clearly not trustable. Fortunately, the optimization procedure was not affected by this phenomena and succeeded to return a local maximum in the vicinity of $\alpha^*$. 
However, another peculiar behaviour occurs in this situation, which is the small M-shape of the approximate likelihood in this vicinity. This feature was common to most of the approximate likelihoods in our simulations on $W=[0,3]^2$ with $\rho^*=100$ and $\alpha^*=0.05$, but we are not able to provide a clear explanation of this phenomena. The consequence is that the optimizer chooses one of the two local maxima from this M-shape, resulting in a bi-modal distribution of $\hat\alpha$ in this case, as showed in the leftmost plot of Figure~\ref{Bessel issues}. This also explains the shape of the boxplot associated to this case in Figure~\ref{Boxplots Bessel}. In front of such peculiar M-shape of the contrast function, it might be natural to choose as the optimum the average of the two local maxima instead of one of them. Adopting this strategy decreases the estimation mean square error from $1$ to $0.25$ (x$10^{-6}$).

\begin{figure}
\centering
\begin{tabular}{cc}
     \includegraphics[height=0.35\textwidth]{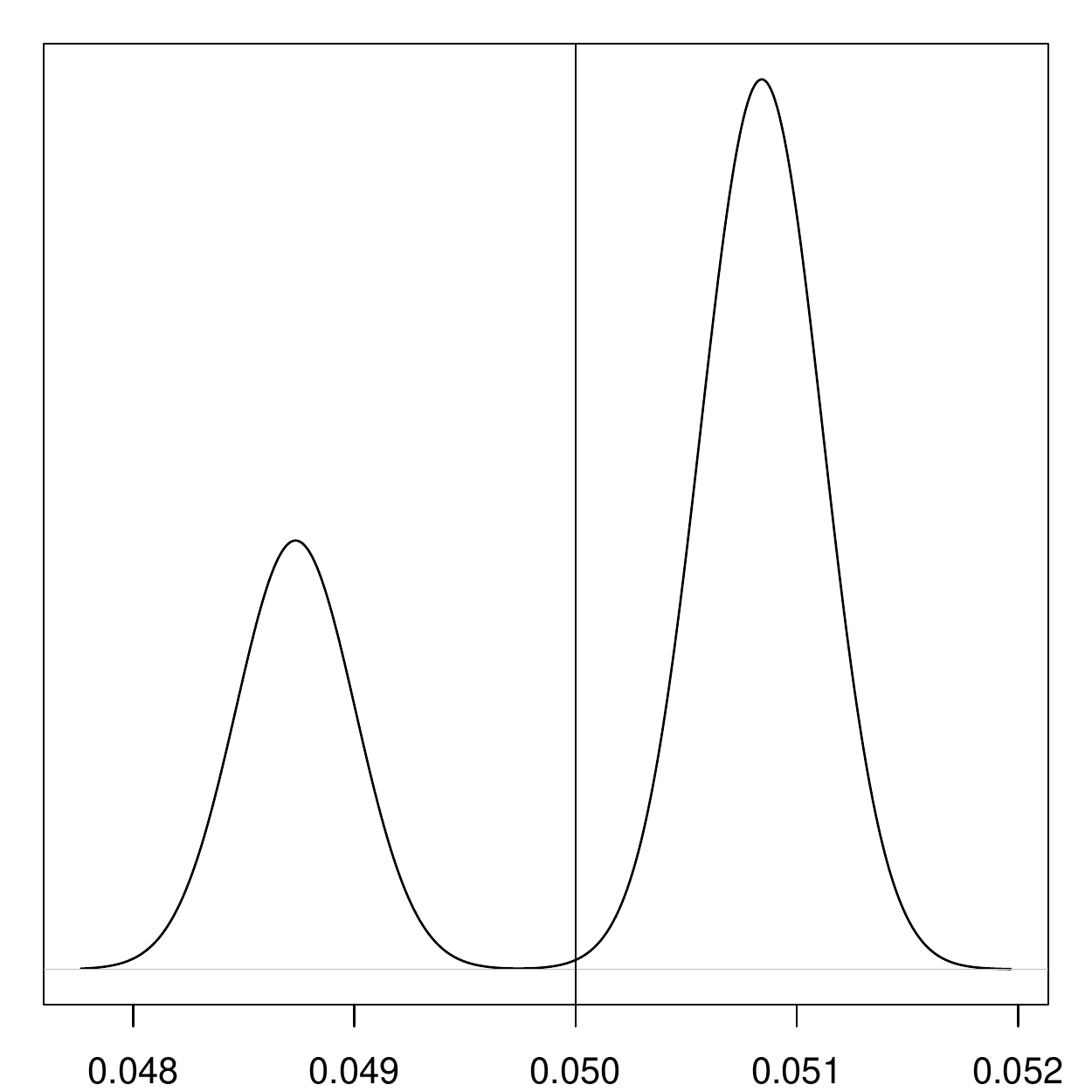}  &
     \includegraphics[height=0.35\textwidth]{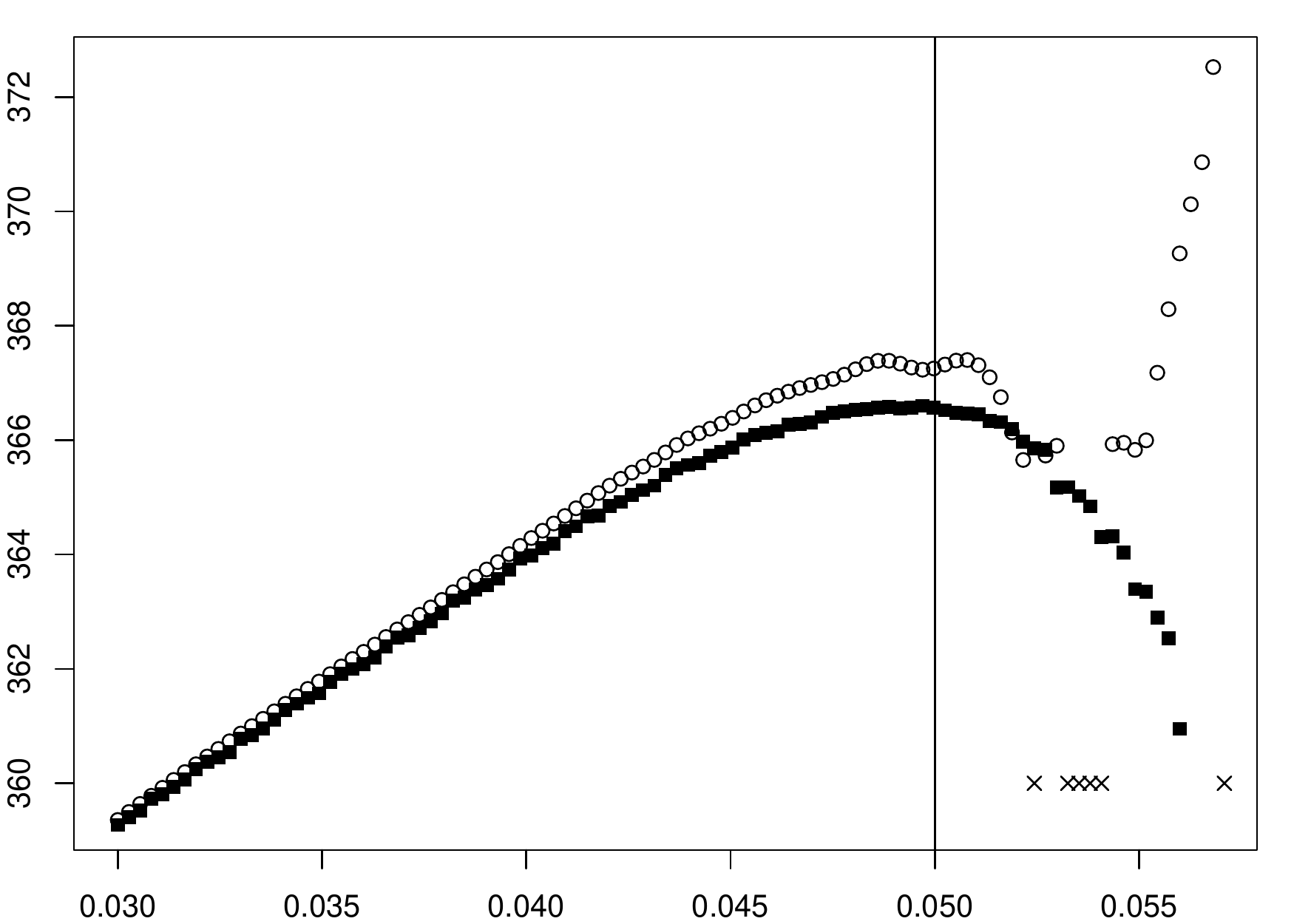} 
   \end{tabular}
	  \caption{\label{Bessel issues} Left: distribution of $\hat\alpha$ obtained by the approximate MLE   $\tilde l^{~\T}$, based on 500 simulations when $\rho^*=100$, $\alpha^*=0.05$ (represented by the vertical line) and $W=[0,3]^2$. Right: $\alpha\mapsto \tilde l^{~\T}(\hat\rho,\alpha|X)$ (circles) and Fourier series approximation \eqref{eq:Fourier_approx} of the log-likelihood (black squares) from one realization $X$ as before, where the vertical line shows the true parameter $\alpha^*=0.05$ and the cross-type points indicate the values of $\alpha$ for which the determinant in  $\tilde l^{~\T}(\hat\rho,\alpha|X)$ was negative.}
\end{figure}

It is interesting to note  that for Bessel-type DPPs, unlike the DPP models of Section~\ref{sec:Simu_Pas_Bessel}, the Fourier series approximation \eqref{eq:Fourier_approx} of the MLE has a more significative difference of behaviour than our approximate MLE with edge correction \eqref{approxLLtorus}. As shown in Figure \ref{Bessel issues}, it does not have undefined values and it does not follow a chaotic behavior for large values of $\alpha$. Moreover, because the Fourier transform of the Bessel kernel only takes two different values (see Table \ref{parametric models}), the terms in the Fourier approximation \eqref{eq:Fourier_approx} when $d=2$ simplify as:
$$\sum_{\underset{\|k\|<N}{k\in\Z^2}}\log(1-\hat K_0^\theta(k))=\log\left(1-\rho\pi\alpha^2\right)\sum_{k\in\{-N,\cdots,N\}} \left(2\left\lfloor\sqrt{\frac{1}{\pi^2\alpha^2}-k^2}\right\rfloor +1\right),~\mbox{and}$$
$$L^\theta_{app}(x,y)=\frac{\rho\pi\alpha^2}{1-\rho\pi\alpha^2}\sum_{k\in\{-N,\cdots,N\}}\frac{\cos(2\pi k x)\sin\left(\pi y\left(2\left\lfloor\sqrt{\frac{1}{\pi^2\alpha^2}-k^2}\right\rfloor +1\right)\right)}{\sin\left(\pi y\right)},$$
where the truncation constant is $N =\left\lfloor \frac{1}{\pi\alpha} \right\rfloor$. This simplification makes it easier to compute than in the general case, and results in a more competitive  computation time, similar to our approximation   \eqref{approxLLtorus}. As a result, we observe in Table~\ref{MSEtable Bessel} and Figure~\ref{Boxplots Bessel}  that for $\alpha^*=0.01$ and $0.03$, the Fourier approximation method has very similar performances than our approximation  \eqref{approxLLtorus}. When $\alpha^*=0.05$, the Fourier approximation estimator has also a similar quadratic error, but the distribution of the estimator is more regular, for the reasons noticed above.

Finally, despite the fact that Bessel-type DPPs are not covered by our theory and the peculiar behaviour of $\tilde l^{~\T}$ for some values of $\alpha$, our approach still remains competitive in this case and outperforms standard MCE methods. Nevertheless, because the Fourier approximation \eqref{eq:Fourier_approx} simplifies nicely in this setting and does not show the same chaotic behaviour as  \eqref{approxLLtorus}  for large values of $\alpha$, it seems to be a slightly better choice for Bessel-type DPPs. However, we recall that this approach is limited to rectangular observation windows only.

\subsection{Simulations on a non-rectangular window}\label{sec: simuR}

We consider in this section the estimation of a Gaussian-type DPP on the (non-rectangular) R-shape window as in the simulations of Figure~\ref{fig:simuR}. The underlying parameters are $\rho^*=100$, resulting  in 370 points on average, and $\alpha^*=0.01,0.03$ and $0.05$. The estimation of $\alpha^*$ is carried out by the MLE approximation \eqref{approxLL}  (without edge-corrections), the edge-corrected version described below, and the MCEs based on the pcf and the Ripley's $K$-function. Note that in this situation, the Fourier approximation  \eqref{eq:Fourier_approx} is not feasible. 

\begin{figure}
\centering
\begin{tabular}{ccc}
   \includegraphics[width=0.3\textwidth]{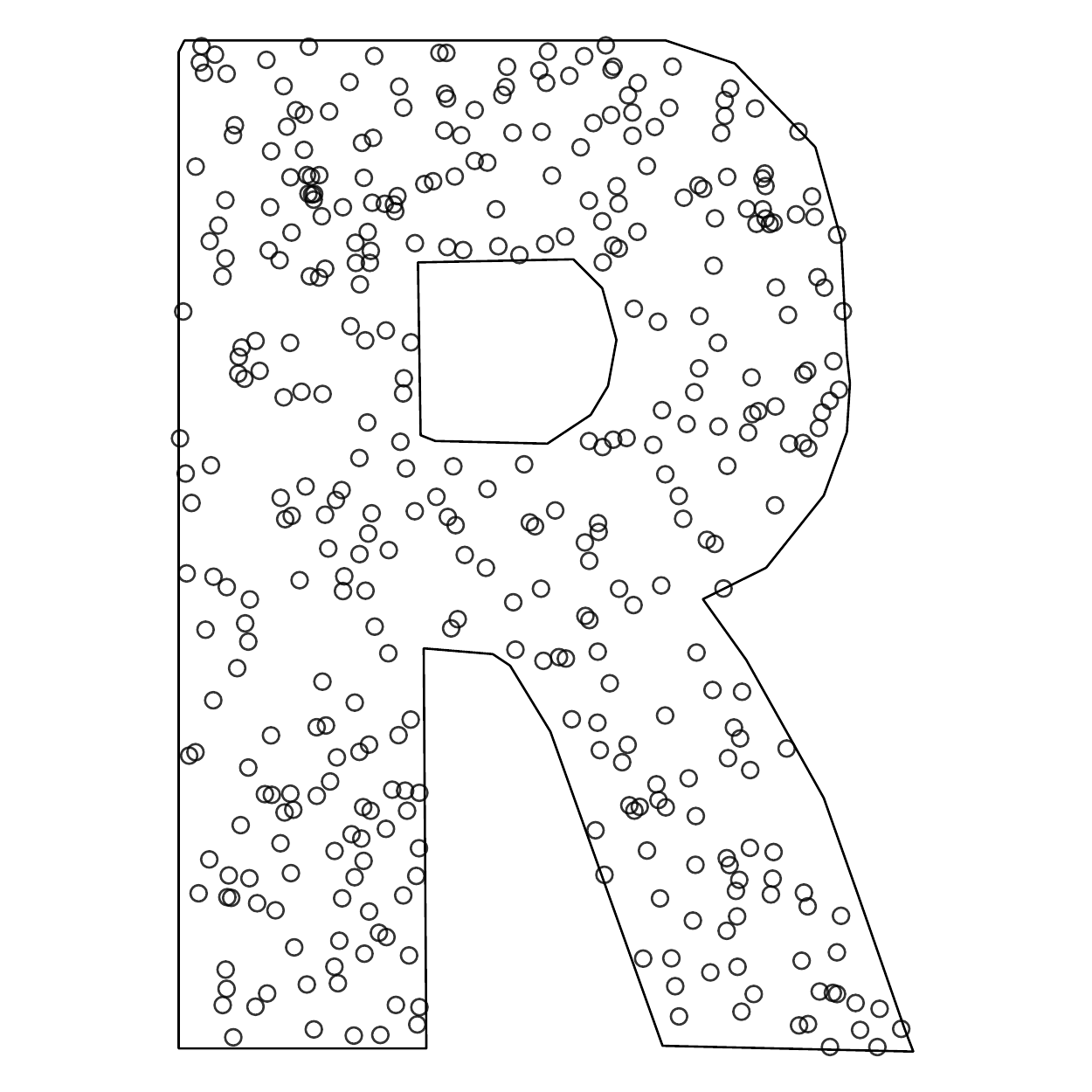} &
      \includegraphics[width=0.3\textwidth]{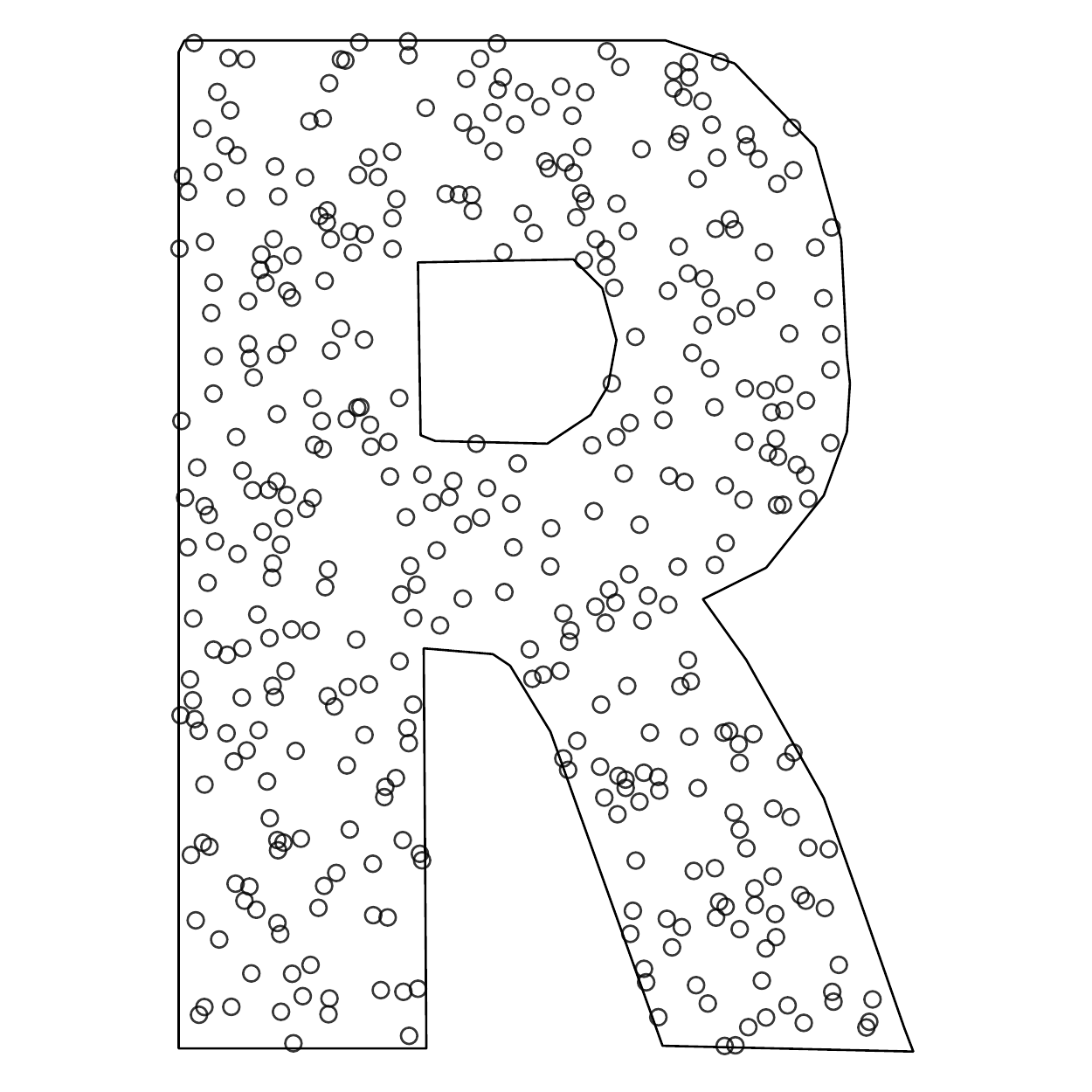} &
   \includegraphics[width=0.3\textwidth]{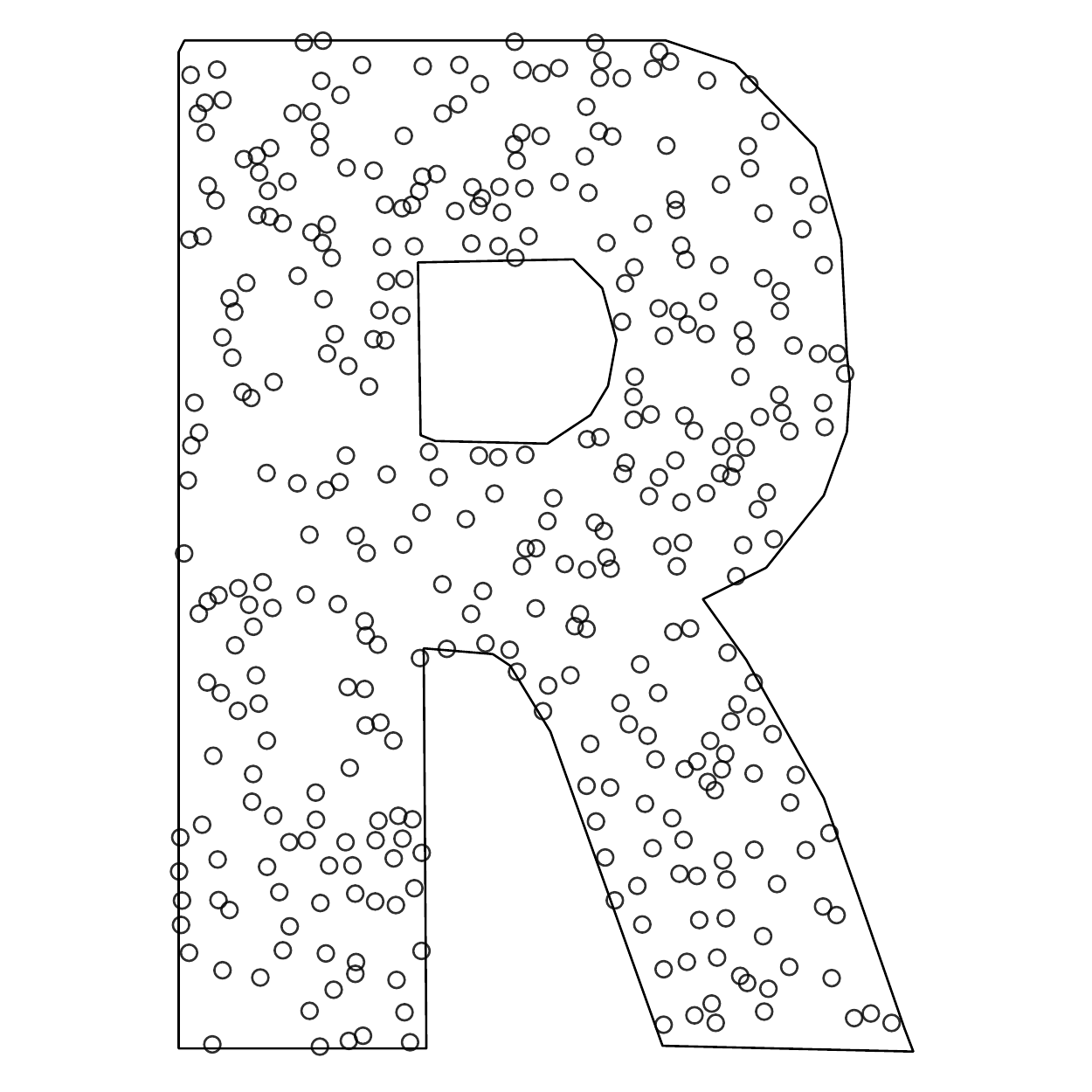} \\
     \includegraphics[width=0.3\textwidth]{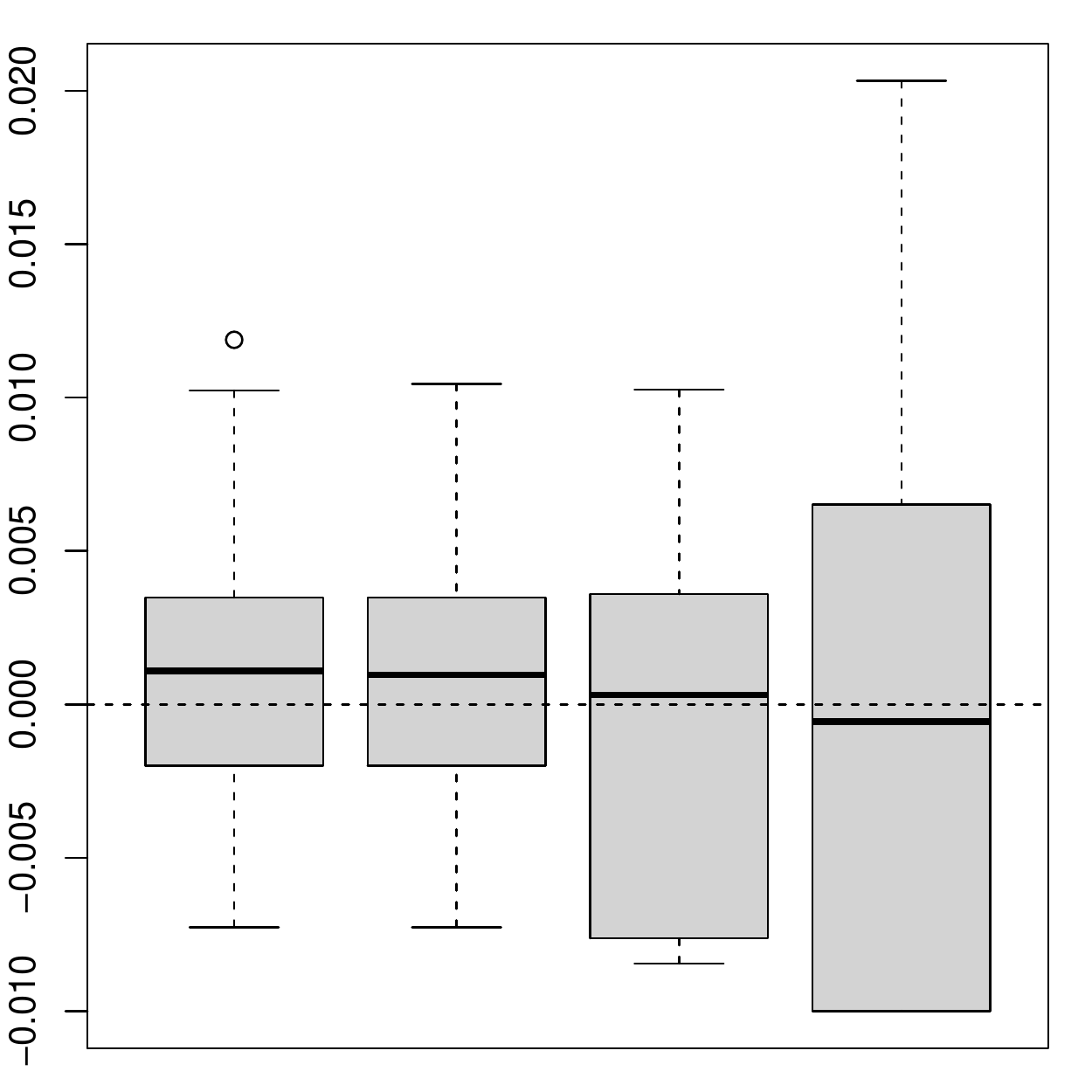} &
      \includegraphics[width=0.3\textwidth]{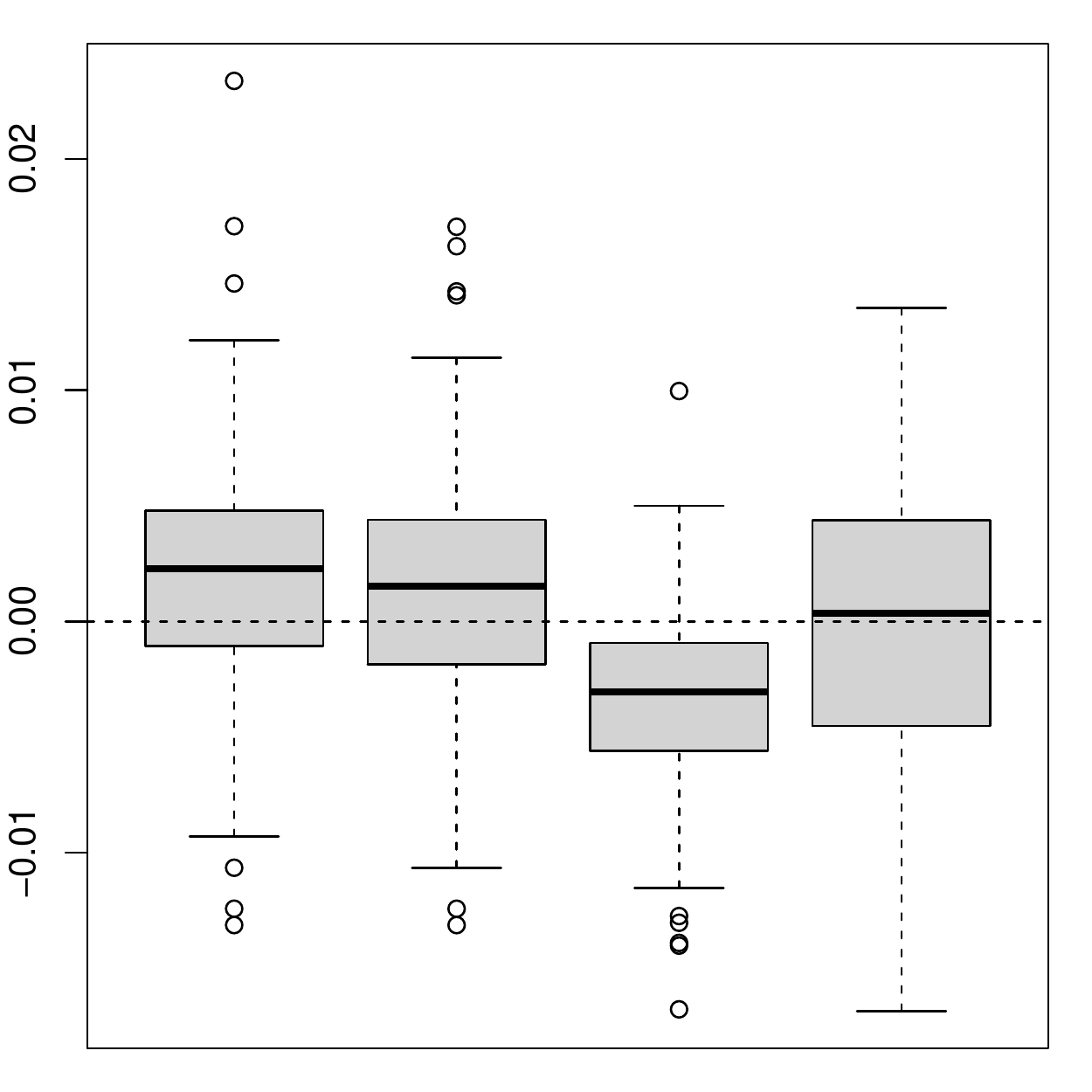} &
   \includegraphics[width=0.3\textwidth]{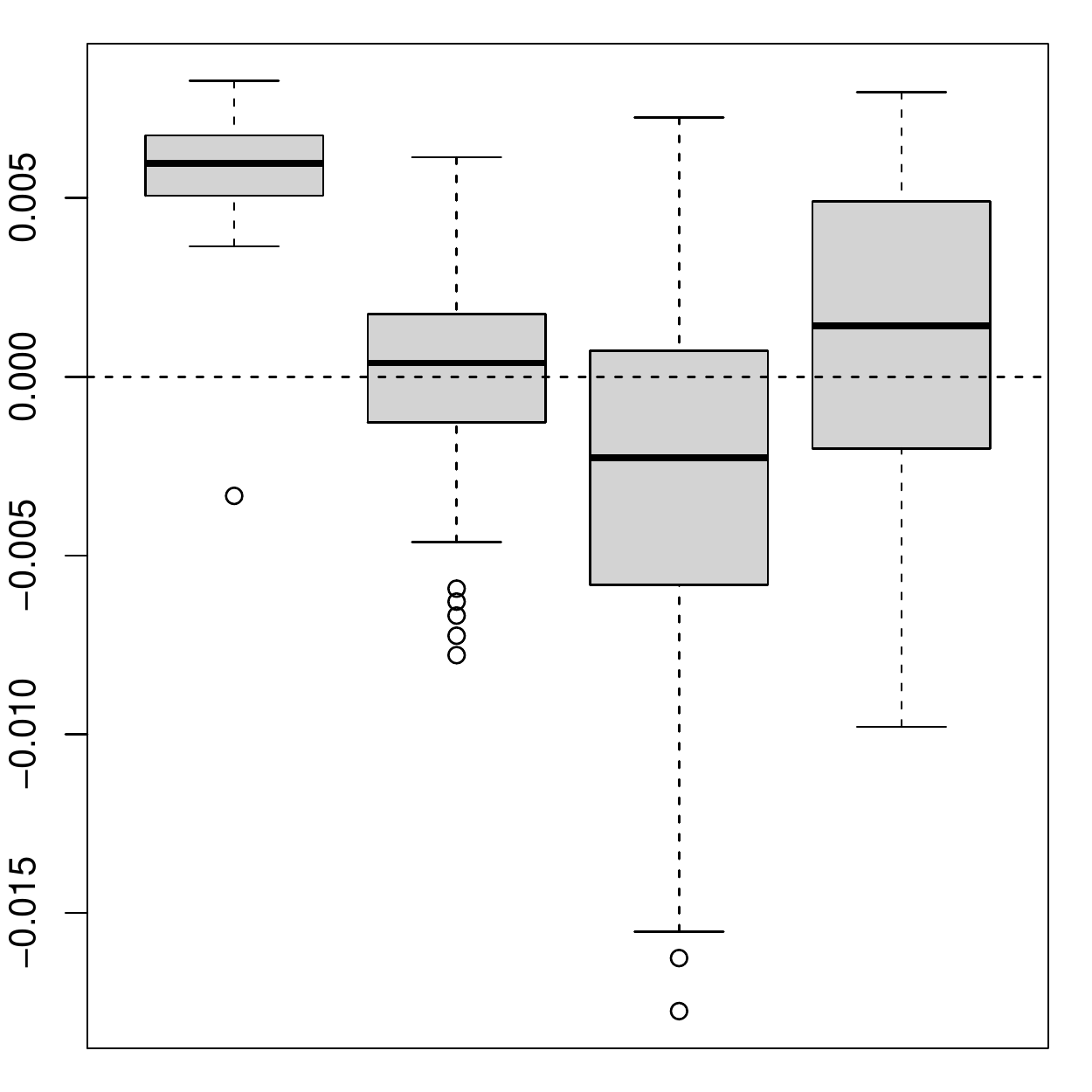}\\
   $\alpha^*=0.01$ & $\alpha^*=0.03$ & $\alpha^*=0.05$ 
   \end{tabular}
	   \caption{\label{fig:simuR} Top: Examples of realizations of Gaussian-type DPPs with parameters $\rho^*=100$ and $\alpha^*=0.01,0.03,0.05$ (from left to right) on a R-shape window. Bottom: distribution of $\hat\alpha$ from 100 simulations for each value of $\alpha^*$ and for the following estimators (from left to right in each plot): the approximate MLE \eqref{approxLL}, its edge-corrected version as detailed in the text, and the MCEs based on the pcf and the Ripley's $K$ function.}
\end{figure}

We handle the edge-effects for this non-rectangular window in the following way.  Note that the periodic edge-correction presented in Section~\ref{sec:periodic_correction} amounts to replace some zero-values of the matrix $L_0^\theta[X\cap W]$ by non negligible values. If we assume that the function $L_0^\theta$ is isotropic, as for the Gaussian-type DPPs considered in this section, then $L_0^\theta[X\cap W]=L_\rad^\theta(R)$ where $R$ is the pairwise distance matrix of $X$, i.e. $R=(r_{ij}:=\|x_i-x_j\|)_{1\leq i,j\leq n}$ if $X=\{x_1,\dots,x_n\}$. Precisely, the replacements concern the entries   involving a point $x_i$ close to the border of the window and they amount to replace some of the largest distances $r_{ij}$ by smaller ones. The idea is that for these points, we need to artificially increase the number of closed neighbours to account for edge-effects. Adopting this idea, we replace some entries of the matrix $R$ as follows:
\begin{enumerate}
\item We start by setting a maximal range of interaction $r_{\max}$. In our example we choose  $$r_{\max}=\argmax_{r_{ij}}\{L_{\rad}^{+} (r_{ij})>0.001 L_{\rad}^{+}(0)\},$$ where $L_{\rad}^+=L_{\rad}^\theta$ for $\theta=(\hat\rho, 0.9\alpha_{\max})$, $\hat\rho=N(W)/|W|$, $\alpha_{\max}=\sqrt{1/(\pi \hat \rho)}$ and $L_{\rad}^+(0)$ is the maximal possible value of $L_{\rad}^+$. This choice guarantees that for any $r>r_{\max}$ and any $\theta=(\hat\rho,\alpha)$ with $\alpha<0.9\alpha_{\max}$, $L_{\rad}^\theta(r)$ can be considered to be negligible. 
\item For $i=1,\dots,n$, we denote by $d_i$ the Euclidean distance from $x_i$ to $\partial W$, and by $n_i=\card\{j,\,  r_{ij}<r_{\max}\}$ the number of neighbours of $x_i$ in $X$. We further denote by $\mathcal B=\{x_i\in X,\, d_i<r_{\max}\}$ the set of ``border'' points of $X$ in $W$ and by $\bar{\mathcal B}=X\setminus\mathcal B$ the set of ``interior'' points of $X$ in $W$. Finally, we consider $\mathcal R_{\bar{\mathcal B}}=\{r_{ij},\, x_i\in \bar{\mathcal B}\}$ the set of observed pairwise distances for the interior points of $X$, and $\mathcal N_{\bar{\mathcal B}}=\{n_i, x_i\in \bar{\mathcal B}\}$ the set of numbers of neighbours of the interior points.
\item For all $x_i\in\mathcal B$, we randomly pick out $\tilde n_i$ in $\mathcal N_{\bar{\mathcal B}}$ and compare it to $n_i$. 
If $n_i\geq \tilde n_i$, we do nothing. Else, for $j=(i+1),\dots,(i+\tilde n_i-n_i)\wedge n$ and if $r_{ij}>r_{\max}$, we randomly pick out $\tilde r_{ij}$ in $\mathcal R_{\bar{\mathcal B}}\cap\{r_{ij}>d_i\}$ and we replace $r_{ij}$ and $r_{ji}$ by $\tilde r_{ij}$. 
\end{enumerate}
Note that the number of replacements in this edge-correction procedure is limited: they  only concern the border points of $X$, there are a maximum of $\tilde n_i-n_i$ of them for each border point $x_i$, and the replaced value $\tilde r_{ij}$ of $r_{ij}>r_{\max}$ is necessarily greater than $d_i$, which in many cases (especially if $\alpha$ is small) entails $L_{\rad}^\theta(\tilde r_{ij})\approx 0$ and does not affect the initial value $L_{\rad}^\theta(r_{ij})\approx 0$. With the resulting new matrix $R$, there is not guaranty that  $L_\rad^\theta(R)$ is positive, a common issue with the periodic edge corrections of 
Section~\ref{sec:periodic_correction}, but the restricted number of replacements limits the risk to encounter such a problem. In our experience, this happened only for very high values of $\alpha$ and did not affect the optimisation procedure. 

The results displayed in Figure~\ref{fig:simuR} show that the above edge-correction version of \eqref{approxLL} provides the best results and clearly outperforms the MCE methods. They also confirm that this edge-correction is only necessary for the most repulsive DPPs, i.e.  $\alpha^*=0.05$ here, otherwise the approximation \eqref{approxLL}  and its edge-corrected version perform just as well.

\subsection{Estimation of the standard errors}

In order to numerically assess the quality of estimation of the standard errors, as described in Section~\ref{sec:sd}, 
we consider the estimation of $\alpha^*$ for Gaussian, Cauchy and Bessel families of DPPs when the observation window $W$ is either $[0,1]^2$, $[0,2]^2$ or $[0,3]^2$, the intensity is $\rho^*=100$ and for three different values of $\alpha^*$. These values of $\alpha^*$ correspond to low, mild and strong repulsion, specifically $\alpha^*\in\{0.01,0.03,0.05\}$  for the Gaussian and Bessel models and $\alpha^*\in\{0.005,0.02,0.035\}$ for the Cauchy model. 

In these cases, we have $\theta=(\rho,\alpha)$ and the observed information matrix is estimated by 
$$\tilde I(\hat\rho,\hat\alpha) =  -|W|\begin{pmatrix}
\partial_{\rho}^2 \tilde l^{~\T}(\hat\rho, \hat\alpha|X) & \partial_{\rho}\partial_{\alpha} \tilde l^{~\T}(\hat\rho, \hat\alpha|X)\\
\partial_{\alpha}\partial_{\rho} \tilde l^{~\T}(\hat\rho, \hat\alpha|X) & \partial_{\alpha}^2 \tilde l^{~\T}(\hat\rho, \hat\alpha|X)\end{pmatrix},$$
where $\hat\rho=N(W)/|W|$, $\hat\alpha\in\argmin_{\alpha}{\tilde l^{~\T}(\hat\rho, \alpha|X)}$,  $\tilde l^{~\T}$ is given by \eqref{approxLLtorus}, and the derivatives are obtained using \eqref{eq:2nd_derivative_LL}, Table \ref{parametric models} and Table \ref{L0 models}.

Following Section~\ref{sec:sd}, the variance of $(\hat\rho,\hat\alpha)$ is estimated by $\tilde I(\hat\rho,\hat\alpha)^{-1}$ and an approximated $95\%$ confidence interval for $\alpha^*$ is then
\begin{equation}\label{eq:IC}
\left[\hat\alpha \pm \frac{1.96}{\sqrt{|W|}}\left(\frac{(\partial_{\alpha}\partial_{\rho} \tilde l^{~\T}(\hat\rho, \hat\alpha|X))^2}{\partial_{\rho}^2 \tilde l^{~\T}(\hat\rho, \hat\alpha|X)}-\partial_{\alpha}^2 \tilde l^{~\T}(\hat\rho, \hat\alpha|X)\right)^{-1/2}\right].
\end{equation}

In Table \ref{tab:IC_correct}, we report, for each DPP family, each choice of window $W$ and each value of $\alpha^*$, the proportion of times $\alpha^*$ falls in that interval, out of $500$ simulations. Note that it might happen that \eqref{eq:IC} is not well-defined, which occurs when $\tilde I(\hat\rho,\hat\alpha)$ is not  positive definite. We report in parenthesis in Table \ref{tab:IC_correct} the proportion of times this issue arose for each case.

We observe from Table~\ref{tab:IC_correct} that the approximated confidence interval \eqref{eq:IC} is inaccurate for the Bessel family,  which is in line with our findings of Section~\ref{sec bessel}. In particular, the peculiar behavior of $\alpha\mapsto \tilde l^{~\T}(\hat\rho,\alpha|X)$ in the vicinity of $\alpha^*$ when $\alpha^*=0.05$ and $W=[0,3]^2$, as shown in Figure~\ref{Bessel issues}, makes irrelevant any estimation of its derivatives at $\alpha^*$, which certainly explains the low coverage rate of  the interval \eqref{eq:IC}  observed in this case. For the Gaussian and Cauchy DPPs families, the results are all the more satisfying that the window is large. For these families, the approximated confidence interval  \eqref{eq:IC} seems to be  trustable whenever there are more than 400 points (corresponding to the case $W=[0,2]^2$), even if it tends to seemingly underestimate the standard error of $\hat\alpha$. 

We finally made the same kind of simulations for the (non-rectangular) R-shape window of Section~\ref{sec: simuR} with the same model as there. Note that the setting is comparable with the Gaussian DPP when $W=[0,2]^2$ in Table~\ref{tab:IC_correct}, except for the shape of the window, since the mean number of points are similar and  the kernel and the values of $\alpha^*$ are the same. The coverage rate for the R-shape window was $87.4\%$ for $\alpha^*=0.01$ (low repulsion), $90\%$ for  $\alpha^*=0.03$ (mild repulsion) and $63.2\%$ for $\alpha^*=0.05$ (strong repulsion). These results are of the same order as in Table~\ref{tab:IC_correct} for low and mild repulsion, but worse for $\alpha^*=0.05$. This last result is probably due to edge-effects in this case, and shows that there is still an avenue to improve edge-corrections for non-rectangular windows. 

\begin{table}[H]
\centering
{\renewcommand{\arraystretch}{1.25}
\begin{tabular}{|c||c|c|c||c|c|c||c|c|c|}
\hline
Window & \multicolumn{3}{c||}{$[0,1]^2$} & \multicolumn{3}{c||}{$[0,2]^2$} & \multicolumn{3}{c|}{$[0,3]^2$} \\
\hline
$\alpha^*$ & low & mild & high & low & mild & high & low & mild & high \\
\hline
Gauss & 88.2 & 89.6 & 92 & 88.6 & 94.2 & 92.6 & 93.2 & 93.2 & 92.8\\
&  (0) & (0) & (0) & (0) & (0) & (0) & (0) & (0) & (0) \\ \hline
Cauchy & 89.8 & 88.4 & 72 & 92.4 & 92.2 & 83.6 & 91.4 & 95 & 87.2 \\
& (2) & (3.6) & (25.4) & (0) & (0) & (0) & (0) & (0) & (0)  \\ \hline
Bessel & 66 & 76.6 & 56 & 77.2 & 78.8 & 82.2 & 81 & 74.4 & 12\\
& (0) & (0.2) & (16.2) & (0) & (0) & (7.6) & (0) & (0) & (2.2)\\
\hline
\end{tabular}}
\caption{Proportion of simulations ($\times 100$), out of $500$, in which the approximated $95\%$ confidence interval \eqref{eq:IC} is well-defined and contains $\alpha^*$, for the Gaussian, Cauchy and Bessel DPPs' models, three different observation windows and three values of $\alpha^*$ corresponding to low, mild and strong repulsion (see the text for the exact values).  In parenthesis are indicated the proportion of simulations ($\times 100$)  when the interval \eqref{eq:IC} was not well-defined.} \label{tab:IC_correct}
\end{table}

\section{Conclusion} \label{sec:Conclusion}

In this paper, we have introduced an asymptotic approximation \eqref{approxLL} of the log-likelihood of stationary determinantal point processes on $\R^d$ and $\Z^d$. While the true likelihood is not numerically tractable, this approximation can be computed for stationary parametric families of DPPs based on correlation functions with a known Fourier transform, as the classical ones presented in Table~\ref{parametric models}. Compared to the Fourier approximation of \cite{Lavancier} that only works for rectangular windows, our approximation can be computed for windows of any shape and provides an estimation of the asymptotic variance of the resulting maximum likelihood estimator. However, due to edge effects, the resulting estimator gets heavily biased for strongly repulsive DPPs, as shown in Figure \ref{Boxplots}. We have proposed to use the periodic correction \eqref{approxLLtorus} to fix this issue in the case of rectangular windows and showed that the resulting approximation is very close to the one in \cite{Lavancier} (see Figure \ref{ComparingSmoothness}) but overall easier to compute. The idea to use a periodic correction has been detailed for rectangular windows, but a similar idea can be applied for a window with a different shape, as exemplified in Section~\ref{sec: simuR}.
We showed in the simulation study of Section \ref{sec:Simu} that for standard parametric families of DPPs, the resulting approximate MLE outperforms classical moment methods based on the pair correlation function and Ripley's $K$ function.

Finally, we proved in Propositions \ref{DeterministicCV} and \ref{approxLLCV} that the difference between the true log-likelihood and our approximated log-likelihood converges almost surely towards $0$ for classical parametric families of stationary DPPs on $\Z^d$. We also showed in Proposition \ref{GridApprox2} that DPPs on $\R^d$ can be arbitrarily approached by DPPs on a regular grid, which suggests that our approximation should also converge for DPPs on $\R^d$.  A formal proof of such  result is still a seemingly difficult open problem.
Beyond the approximation of the likelihood, as proposed in this paper, a natural theoretical concern is  the consistency of the maximum likelihood estimator, either based on the true likelihood or on the approximated one. This question is challenging and is not addressed in the present contribution. We however think that our findings  are a step in the right direction towards such a result, because they allow to replace the true likelihood by an easier expression to deal with mathematically.

\section{Proofs of Section \ref{SEC:DEFLL}} \label{sec:ProofCons}

\subsection{Proof of Proposition \ref{DeterministicCV}} \label{DetPart}
In the case where $\X=\R^d$ and $W_n$ is of the form $n\times W$ for some compact set $W$, then the convergence of $\frac{1}{|W_n|}\logdet(\mathcal{I}_{W_n}-\mathcal{K}^{\theta}_{W_n})$ for any fixed value of $\theta$ corresponds to \cite[Proposition 5.9]{Shirai} with $\alpha=-1$ and $f(x)=\infty \times \cara{W}$. Our proof follows a similar idea.

Since all eigenvalues of $\mathcal{K}^{\theta}_{W_n}$ are in $[0,1[$ then the logarithm of the Fredholm determinant of $\mathcal{I}_{W_n}-\mathcal{K}^{\theta}_{W_n}$ can be expanded into
$$\logdet(\mathcal{I}_{W_n}-\mathcal{K}^{\theta}_{W_n})\hspace{-0.05cm}=\hspace{-0.05cm}-\sum_{k\geq 1}\frac{\tr((\mathcal{K}^{\theta}_{W_n})^k)}{k}=-\sum_{k\geq 1}\frac{1}{k}\int_{W_n^k}K^{\theta}_0(x_2-x_1)\cdots K^{\theta}_0(x_1-x_k)\der\nu^k(x).$$
Now, we first assume that $K^{\theta}_0$ and $\hat K^{\theta}_0$ are integrable. Then, for any $x_1\in\X$ the function 
$$f^\theta_{x_1}:(x_2,\cdots,x_k)\mapsto K^{\theta}_0(x_2-x_1)\cdots K^{\theta}_0(x_1-x_k)$$ 
is integrable and its integral is equal to $(K_0^{\theta})^{*k}(0)$ where $(K_0^{\theta})^{*k}(0)$ is the $k$-th times self-convolution of $K_0^{\theta}$. Since we assumed that $(W_n)_{n\geq 0}$ satisfy \eqref{CondWindow}, then by Lemma \ref{CVtransinv} we get that

\begin{multline*}
\left|\frac{1}{|W_n|}\int_{W_n^k}K_0^{\theta}(x_2-x_1)\cdots K_0^{\theta}(x_1-x_k)\der\nu^k(x) - (K_0^{\theta})^{*k}(0)\right|\\
\leq\int_{(\mathcal{B}(0,r_n)^C)^{k-1}}|f^\theta_0(y)|\der \nu^{k-1}(y)+\frac{|(\partial W_n\oplus r_n)\cap W_n|}{|W_n|}\|f^\theta_0\|_{L^1}
\end{multline*}
for any positive sequence $(r_n)_{n\in\N}$. Now, since we assumed that $x\mapsto \sup_{\theta\in\Theta} K_0^\theta(x)$ is integrable then $x\mapsto \sup_{\theta\in\Theta} f_0^\theta(x)$ is also integrable. Moreover, the sequence $(W_n)_{n\in\N}$ satisfies condition $(\mathcal{W})$, hence
$$\sup_{\theta\in\Theta}\left|\frac{1}{|W_n|}\int_{W_n^k}K_0^{\theta}(x_2-x_1)\cdots K_0^{\theta}(x_1-x_k)\der\nu^k(x) - (K^{\theta}_0)^{*k}(0)\right|\cvn 0.$$
Finally, since
$$\frac{|\tr((\mathcal{K}^{\theta}_{W_n})^k)|}{k |W_n|}\leq \frac{\|\mathcal{K}^{\theta}\|^{k-1}}{k}\times\frac{\tr(\mathcal{K}^{\theta}_{W_n})}{|W_n|}\leq \frac{M^{k-1}}{k} \sup_{\theta\in\Theta}K^{\theta}_0(0)$$
which is summable with respect to $k$ and does not depend on $n$ and $\theta$, then we can conclude by the dominated convergence theorem that
$$\sup_{\theta\in\Theta}\left|\frac{1}{|W_n|}\logdet(\mathcal{I}_{W_n}-\mathcal{K}^{\theta}_{W_n})+\sum_{k\geq 1} \frac{K_0^{\theta*k}(0)}{k}\right|\cvn 0.$$
The proof is completed by using the relation
$$-\sum_{k\geq 1} \frac{(K_0^{\theta})^{*k}(0)}{k}=\int_{\X^*}\log(1-\hat K^{\theta}_0(x))\der x.$$

\subsection{Proof of Proposition \ref{approxLLCV}} \label{CVtoExp}

For this proof, we consider $X$ to be a DPP on $\Z^d$. Let $\theta\in\Theta$, we denote by $\lambda^\theta_m$ the lowest eigenvalue of $K^\theta[\Z^d]$ and we define $\lambda^0_m\defeq\inf_{\theta\in\Theta} \lambda^\theta_m$. It is important to note that $\lambda^0_m>0$ as a consequence of \cite[Theorem 5]{min} and the assumptions on $K^\theta$. We begin by proving the following lemma allowing us to control $L^\theta[X\cap W]-L_{[W]}^\theta[X\cap W]$ for any $W\subset\Z^d$ by controlling the difference between their associated operators.

\begin{lem} \label{ineqdet}
Let $\mathcal{K}$ be an integral operator on $L^2(\X,\nu)$ with kernel $K$ such that $\|\mathcal{K}\|<1$. For any Borel set $W\subset\X$, we denote by $\mathcal{P}_W$ the projection on $L^2(W)$ and we define the operators $\mathcal{K}_W\defeq \mathcal{P}_W\mathcal{K}\mathcal{P}_W$ on $L^2(W)$, $\mathcal{L}_{[W]}\defeq \mathcal{K}_W(\mathcal{I}_W-\mathcal{K}_W)^{-1}$ on $L^2(W)$ and $\mathcal{L}\defeq \mathcal{K}(\mathcal{I}-\mathcal{K})^{-1}$ on $L^2(\X)$. We denote by $L$ the kernel of $\mathcal{L}$ and finally we define the operator $\mathcal{N}_{W}\defeq \mathcal{P}_W \mathcal{LP}_{W^C} \mathcal{LP}_W$ on $L^2(W)$ with kernel
\begin{equation}\label{defN} N_{W}(x,y)=\int_{W^C}L(x,z)L(z,y) \der\nu(z)~~\forall x,y\in W.\end{equation}
Then,
$$0\leq\mathcal{P}_W\mathcal{L}\mathcal{P}_W-\mathcal{L}_{[W]}\leq \mathcal{N}_{W}.$$

\end{lem}
\begin{proof}
We consider the following decomposition of the linear operators $\mathcal{I}-\mathcal{K}$ and $(\mathcal{I}-\mathcal{K})^{-1}$ on $L^2(W)\oplus L^2(W^C)$:
$$\mathcal{I}-\mathcal{K}=\left ( \begin{tabular}{ll}
$\mathcal{I}_W-\mathcal{K}_{W}$ & $-\mathcal{P}_W\mathcal{K}\mathcal{P}_{W^C}$ \\
$-\mathcal{P}_{W^C}\mathcal{K}\mathcal{P}_W$ & $\mathcal{I}_{W^C}-\mathcal{K}_{W^C}$ \\
\end{tabular} \right )=\left ( \begin{tabular}{ll}
$(\mathcal{L}_{[W]}+\mathcal{I}_W)^{-1}$ & $-\mathcal{P}_W\mathcal{K}\mathcal{P}_{W^C}$ \\
$-\mathcal{P}_{W^C}\mathcal{K}\mathcal{P}_W$ & $(\mathcal{L}_{[W^C]}+\mathcal{I}_{W^C})^{-1}$ \\
\end{tabular} \right )$$
and
$$(\mathcal{I}-\mathcal{K})^{-1}=\mathcal{I}+\mathcal{L}=\left ( \begin{tabular}{ll}
$\mathcal{P}_W\mathcal{L}\mathcal{P}_W+\mathcal{I}_W$ & $\mathcal{P}_W\mathcal{L}\mathcal{P}_{W^C}$ \\
$\mathcal{P}_{W^C}\mathcal{L}\mathcal{P}_W$ & $\mathcal{P}_{W^C}\mathcal{L}\mathcal{P}_{W^C}+\mathcal{I}_{W^C}$ \\
\end{tabular} \right ).$$
A well-known result is that the $(1,1)$ block of $\mathcal{I}-\mathcal{K}$ is equal to the inverse of the Schur complement of $(\mathcal{I}-\mathcal{K})^{-1}$ relative to its $(2,2)$ block. This property is proved for $2\times 2$ block matrices in  \cite[Theorem 1.2]{Schur}, and  since the proof does not use any finite dimensionality argument, it works all the same for nonsingular operators on a Hilbert space, see \cite{Nonsub} for example. 
As a consequence, we get
$$(\mathcal{L}_{[W]}+\mathcal{I}_W)^{-1}=(\mathcal{P}_W\mathcal{L}\mathcal{P}_W+\mathcal{I}_W-\mathcal{P}_W\mathcal{L}\mathcal{P}_{W^C}(\mathcal{P}_{W^C}\mathcal{L}\mathcal{P}_{W^C}+\mathcal{I}_{W^C})^{-1}\mathcal{P}_{W^C}\mathcal{L}\mathcal{P}_W)^{-1}$$
hence
$$\mathcal{P}_W\mathcal{L}\mathcal{P}_W-\mathcal{L}_{[W]} = \mathcal{P}_W\mathcal{L}\mathcal{P}_{W^C}(\mathcal{P}_{W^C}\mathcal{L}\mathcal{P}_{W^C}+\mathcal{I}_{W^C})^{-1}\mathcal{P}_{W^C}\mathcal{L}\mathcal{P}_W\geq 0.$$
Finally, since $(\mathcal{P}_{W^C}\mathcal{L}\mathcal{P}_{W^C}+\mathcal{I}_{W^C})^{-1}\leq \mathcal{I}_{W^C}$ this concludes the lemma.
\end{proof}
Now, we rewrite 
$$\frac{1}{|W_n|}\big|\logdet(L^\theta[X\cap W_n])-\logdet(L^\theta_{[W_n]}[X\cap W_n])\big|$$
as
$$\frac{1}{|W_n|}\big|\logdet\big(Id+(L^\theta[X\cap W_n]-L_{[W_n]}^\theta[X\cap W_n])L_{[W_n]}^\theta[X\cap W_n]^{-1}\big)\big|.$$
By Lemma \ref{ineqdet}, we know that
$$0\leq L^\theta[X\cap W_n]-L_{[W_n]}^\theta[X\cap W_n]\leq N_{W_n}^\theta[X\cap W_n]$$
where $N_{W_n}^\theta$ is defined as in \eqref{defN}. Therefore, using Lemma \ref{stupidlem} we obtain the bound
$$0\leq\logdet(L^\theta[X\cap W_n])-\logdet(L^\theta_{[W_n]}[X\cap W_n])\leq\tr(N^\theta_{W_n}[X\cap W_n]L^\theta_{[W_n]}[X\cap W_n]^{-1}).$$
Now, since $\mathcal{L}^\theta_{[W_n]}\geq\mathcal{K}^\theta_{W_n}$ by definition, then $\lambda_{min}(L^\theta_{[W_n]}[X\cap W_n])\geq\lambda_{min}(K^\theta[X\cap W_n])\geq \lambda_m^\theta\geq \lambda_0^\theta$ where the second to last inequality is a consequence of $K^\theta[X\cap W_n]$ being a sub-matrix of $K^\theta[\Z^d]$. Therefore,
$$\tr(N^\theta_{W_n}[X\cap W_n]L^\theta_{[W_n]}[X\cap W_n]^{-1})\leq (\lambda_0^\theta)^{-1}\tr(N^\theta_{W_n}[X\cap W_n])=(\lambda_0^\theta)^{-1}\sum_{x\in X\cap W_n}N^\theta_{W_n}(x,x).$$
The function $X\mapsto|W_n|^{-1}\sup_{\theta\in\Theta}\sum_{x\in X}N^\theta_{W_n}(x,x)$ is Lipschitz continuous on $\bigcup_{k\geq 0}W_n^k$ with constant $\sup_{\theta\in\Theta}\|N^\theta_{W_n}\|_{\infty}/|W_n|$ where
$$\sup_{\theta\in\Theta}\|N^\theta_{W_n}\|_{\infty}\leq\sup_{\theta\in\Theta}\|L^\theta_0\|_2^2=\sup_{\theta\in\Theta}\left\|\frac{\hat K^\theta_0}{1-\hat K^\theta_0}\right\|_2^2\leq\frac{\sup_{\theta\in\Theta}\|\hat K^\theta_0\|_2^2}{1-M}=\frac{\sup_{\theta\in\Theta}\|K^\theta_0\|_2^2}{1-M}.$$
This expression is finite since we assumed \eqref{assumptions}. By \cite[Theorem 3.5]{Pemantle}, we then get for all $a\in\R_+$
\begin{multline} \label{merciPemPer}
\mathbb{P}_{\theta^*}\left(\frac{1}{|W_n|}\left|\sup_{\theta\in\Theta}\sum_{x\in X\cap W_n}N^\theta_{W_n}(x,x)-\E_{\theta^*}\left[\sup_{\theta\in\Theta}\sum_{x\in X\cap W_n}N^\theta_{W_n}(x,x)\right]\right|>a\right)\\
\leq 5\exp\left(-\frac{a^2|W_n|^2/\sup_{\theta\in\Theta}\|N^\theta_{W_n}\|_{\infty}^2}{16(a |W_n|/\sup_{\theta\in\Theta}\|N^\theta_{W_n}\|_{\infty}+2\E_{\theta^*}[N(W_n)])}\right)
\end{multline}
where $\E_{\theta^*}[N(W_n)]=|W_n|K_0^{\theta^*}(0)$ and
\begin{align*}
&\frac{1}{|W_n|}\E_{\theta^*}\left[\sup_{\theta\in\Theta}\sum_{x\in X\cap W_n}N^\theta_{W_n}(x,x)\right]\\
\leq & \frac{1}{|W_n|}\E_{\theta^*}\left[\sum_{x\in X\cap W_n}\sup_{\theta\in\Theta}N^\theta_{W_n}(x,x)\right]\\
=&\frac{K_0^{\theta^*}(0)}{|W_n|}\int_{W_n}\sup_{\theta\in\Theta}\int_{W_n^C}L_0^\theta(y-x)^2\der \nu(x)\der \nu(y)\\
\leq&\frac{K_0^{\theta^*}(0)}{|W_n|}\int_{W_n}\int_{W_n^C}\sup_{\theta\in\Theta}L_0^\theta(y-x)^2\der \nu(x)\der \nu(y)\\
=&K_0^{\theta^*}(0)\left(\int_{\Z^d}\sup_{\theta\in\Theta}L_0^\theta(y)^2\der \nu(y)-\frac{1}{|W_n|}\int_{W_n^2}\sup_{\theta\in\Theta}L_0^\theta(y-x)^2\der \nu(x)\der \nu(y)\right).
\end{align*}
But, as a consequence of Lemma \ref{CVtransinv}, we have
$$\frac{1}{|W_n|}\int_{W_n^2}\sup_{\theta\in\Theta}L_0^\theta(y-x)^2\der \nu(x)\der \nu(y)\cvn\int_{\Z^d}\sup_{\theta\in\Theta}L_0^\theta(y)^2\der \nu(y)$$
hence
$$\frac{1}{|W_n|}\E_{\theta^*}\left[\sup_{\theta\in\Theta}\sum_{x\in X}N^\theta_{W_n}(x,x)\right]\cvn 0.$$
Finally, by \eqref{merciPemPer} and the inequality $\|N^\theta_{W_n}\|_{\infty}\leq\|L^\theta_0\|_2^2$, we get that for all $a\in\R_+$,
\begin{multline*}
\mathbb{P}_{\theta^*}\left(\frac{1}{|W_n|}\sup_{\theta\in\Theta}\left|\sum_{x\in X\cap W_n}N^\theta_{W_n}(x,x)\right|>a\right)\\
=O\left(\exp\left(-\frac{a^2|W_n|}{16\sup_{\theta\in\Theta}\|L^\theta_0\|_2^2(a+2K_0^{\theta^*}(0)\sup_{\theta\in\Theta}\|L^\theta_0\|_2^2)}\right)\right).
\end{multline*}
Since we assumed \eqref{notslow}, then by the Borel–Cantelli Lemma,
$$\frac{1}{|W_n|}\sup_{\theta\in\Theta}\sum_{x\in X\cap W_n}N^\theta_{W_n}(x,x)\cvps 0$$
and therefore
$$\frac{1}{|W_n|}\sup_{\theta\in\Theta}\big|\logdet(L^\theta[X\cap W_n])-\logdet(L^\theta_{[W_n]}[X\cap W_n])\big|\cvps 0.$$

\subsection{Proof of Proposition \ref{GridApprox2}}
First, we need to show that $X_\epsilon$ is a well defined DPP for small enough $\epsilon$ by showing that its kernel, the infinite matrix $\epsilon^dK[\epsilon\Z^d]$, is hermitian with eigenvalues in $[0,1]$. Everything is trivial except for showing that the eigenvalues become lower or equal to $1$ as $\epsilon$ vanishes. For every $v=(v_j)_{j\in\Z^d}$ such that $\sum_j |v_j|^2=1$, we define the function
$$\phi(t)=\sum_{j\in\Z^d}v_j e^{2i\pi<j,t>}$$
such that the integral of $|\phi|^2$ on any unit cube is equal to $1$.
Therefore, we can write
\begin{align*}
\langle v,\epsilon^dK[\epsilon\Z^d]v\rangle &=\sum_{j,k\in\Z^d}\epsilon^d v_jv_kK_0(\epsilon(k-j))\\
&=\sum_{j,k\in\Z^d}v_jv_k\int_{\R^d} \hat K_0(t/\epsilon)e^{2i\pi\langle k-j,t\rangle}\der t\\
&=\int_{\R^d} \hat K_0(t/\epsilon)|\phi(t)|^2\der t\\
&\leq\sum_{i\in\Z^d}\sup_{x\in C_i} \hat K_0(x/\epsilon)
\end{align*}
where $C_i$ is the unit cube defined as $[i_1-1/2,i_1+1/2]\times\cdots\times[i_d-1/2,i_d+1/2]$ for all $i=(i_1,\cdots,i_d)\in\Z^d$. By our assumptions on $\hat K_0$, we have $\sup_{x\in C_0} \hat K_0(x/\epsilon)\leq\|\hat K_0\|_\infty\hspace{-0.1cm}<1$ and 
\begin{align*}
\sup_{x\in C_i}\hat K_0(x/\epsilon)&\leq\sup_{\forall j,~x_j\in[i_j-1/2,i_j+1/2]}\frac{A}{1+\epsilon^{-(d+\tau)}(\sum_{j=1}^d x_j^2)^{(d+\tau)/2}}\\
&=\frac{A}{1+\epsilon^{-(d+\tau)}(\sum_{\substack{1\leq j\leq d \\ i_j\neq 0}} (|i_j|-1/2)^2)^{(d+\tau)/2}}
\end{align*}
hence, the sum of all $\sup_{x\in C_i}\hat K_0(x/\epsilon)$ for $i$ of the form $(i_1,\cdots,i_k,0,\cdots,0)$ where $i_1,\cdots,i_k\in\Z\backslash\{0\}$ and $k\in\{1,\cdots,d\}$ is bounded by
\begin{align*}
&\sum_{i_1,\cdots,i_k\in(\Z\backslash\{0\})^k}\frac{A}{1+\epsilon^{-(d+\tau)}(\sum_{j=1}^k (|i_j|-1/2)^2)^{(d+\tau)/2}}\\
\leq & ~\epsilon^{d+\tau}\hspace{-0.8cm}\sum_{i_1,\cdots,i_k\in(\Z\backslash\{0\})^k} \frac{A}{(\sum_{j=1}^k |i_j|^2)^{(d+\tau)/2}}\underset{\epsilon\rightarrow 0}{\longrightarrow} 0.
\end{align*}
By symmetry, this is also true for the sum of all $\sup_{x\in C_i}\hat K_0(x/\epsilon)$ for $i$ with any $k$ non-zero components and $d-k$ zero components when $k\in\{1,\cdots,d\}$. This shows that 
$$\sum_{\substack{i_1,\cdots,i_d\in\Z^d \\ i\neq(0,\cdots,0)}}\sup_{x\in C_i} \hat K_0(x/\epsilon)\underset{\epsilon\rightarrow 0}{\longrightarrow} 0$$
and therefore
$$0\leq \sup_{v:~\sum_j |v_j|^2=1} \langle v,\epsilon^dK[\epsilon\Z^d]v\rangle\leq 1$$
for small enough values of $\epsilon$, and in this case the DPP $X_\epsilon$ is then well defined.

Now, we prove the weak convergence of the discrete DPPs to the continuous one by showing the pointwise convergence of their Laplace functionals (see \cite[Proposition 11.1.VIII]{DV2}). We recall that the Laplace functional of a point process $Y$ is defined as
$$L_Y(f)\defeq \E_Y\left[\exp\left(-\sum_{x\in Y}f(x)\right)\right]$$
for all non-negative continuous function $f$ vanishing outside a bounded set. Let $D$ be a compact set of $\R^d$ and $f:\R^d\rightarrow\R$ be a continuous function vanishing outside $D$. We define the kernel
$$K_f\defeq (x,y)\mapsto\sqrt{1-e^{-f(x)}}K_0(y-x)\sqrt{1-e^{-f(y)}}$$
and call $\mathcal{K}_f$ its associated integral operator. Then, the Laplace transform of the continuous DPP $X$ reads (see \cite{Shirai})
$$L_X(f)=\det(I-\mathcal{K}_f)=\exp\left(-\sum_{n\geq 1}\frac{1}{n}\tr(\mathcal{K}_f^n)\right)$$
and for all $\epsilon$, the rescaled DPP $\epsilon X_\epsilon$ has the same distribution as a DPP on $\epsilon\Z^d$ with kernel $\epsilon^d K_0(y-x)$ hence its Laplace transform reads
\begin{align*}
L_{\epsilon X_\epsilon}(f)=&\det\left (I-\epsilon^dK_f[\epsilon\Z^d]\right )\\
=&\exp\left(-\sum_{n\geq 1}\frac{\epsilon^{dn}}{n}\tr\left(K_f\left[D\cap \epsilon\Z^d\right]^n\right)\right)\\
=&\exp\left(-\sum_{n\geq 1}\frac{1}{n}\left(\epsilon^{dn}\hspace{-0.5cm}\sum_{x_1,\cdots,x_n\in D\cap \epsilon\Z^d}\hspace{-0.5cm}K_f(x_1,x_2)\cdots K_f(x_{n-1},x_n)K_f(x_n,x_1)\right)\right).
\end{align*}
For all $n\geq 1$, we have the convergence of the following Riemann sum on the compact sets $D^n$:
\begin{multline*}
\epsilon^{dn}\hspace{-0.7cm}\sum_{x_1,\cdots,x_n\in D\cap \epsilon\Z^d}\hspace{-0.5cm}K_f(x_1,x_2)\cdots K_f(x_{n-1},x_n)K_f(x_n,x_1)\\
\underset{\epsilon\rightarrow 0}{\longrightarrow}\int_{D^n}K_f(x_1,x_2)\cdots K_f(x_{n-1},x_n)K_f(x_n,x_1)\der x=\tr(\mathcal{K}_f^n).
\end{multline*}
Moreover, we have 
$$\tr \left(\left(\epsilon^dK_f\left[D\cap\epsilon\Z^d\right]\right)^n\right)\leq \lambda_{\max}\left(\epsilon^dK_f\left[D\cap\epsilon\Z^d\right]\right)^{n-1}\tr\left(\epsilon^dK_f\left[D\cap\epsilon\Z^d\right]\right)$$ 
and since $\mathcal{K}_f\leq\mathcal{K}$ then $\lambda_{\max}(\epsilon^dK_f[D\cap\epsilon\Z^d])\leq\lambda_{\max}(\epsilon^dK[D\cap \epsilon\Z^d])\leq\lambda_{\max}(\epsilon^dK[\epsilon\Z^d])$ which we showed was arbitrary close to $\|K_0\|_\infty<1$ for small enough $\epsilon$, then by the dominated convergence theorem we get that
$$L_{\epsilon X_\epsilon}(f)\underset{\epsilon\rightarrow 0}{\longrightarrow}L_X(f)$$
which proves the weak convergence of the distributions of $\epsilon X_\epsilon$ towards the distribution of $X$ when $\epsilon$ goes towards $0$.

\appendix

\section{Technical Lemmas} \label{SEC:APP4}

\begin{lem}\label{stupidlem}
Let $n\in\N$ and $A,B$ be two $n\times n$ positive semi-definite matrices. Then,
$$0\leq\logdet(I+AB)\leq \tr(AB).$$
\end{lem}
\begin{proof}
We first assume that $B$ is the identity matrix. Let $\lambda_1,\cdots,\lambda_n$ be the eigenvalues (with multiplicity) of $A$. Then,
$$0\leq\logdet(I+A)=\sum_{i=1}^n \log(1+\lambda_i)\leq\sum_{i=1}^n \lambda_i=\tr(A).$$
In the general case, Sylvester's determinant identity gives us
$$0\leq\logdet(I+AB)=\logdet(I+A^{1/2}BA^{1/2})\leq \tr(A^{1/2}BA^{1/2})=\tr(AB).$$
\end{proof}

\begin{lem} \label{CVtransinv}
Let $f:\X^k\mapsto\R$ be a translation invariant function such that
$$(x_2,\cdots,x_k)\mapsto f(0,x_2,\cdots,x_k)\in L^1\big(\X^{k-1},\nu^{k-1}\big),$$
and Let $W_n$ be a sequence of increasing compact subsets of $\X$. Then, for any $r>0$, we have
\begin{multline*}
\left|\int_{\X^{k-1}}f(0,x_2,\cdots,x_k)\der\nu(x_2)\cdots\der\nu(x_k)-\frac{1}{|W_n|}\int_{W_n^k}f(x)\der\nu^k(x)\right|\\
\leq \int_{(\mathcal{B}(0,r)^C)^{k-1}}|f(0,y)|\der \nu^{k-1}(y)+\frac{|(\partial W_n\oplus r)\cap W_n|}{|W_n|}\|f(0,.)\|_{L^1},
\end{multline*}
where $\mathcal{B}(0,r)^C$ is the complement of the euclidian ball centered at the origin with radius $r$.\\
In particular, if there exists a sequence $(r_n)_{n\geq 0}$ satisfying
\begin{equation}\label{borderWn}
|(\partial W_n\oplus r_n)\cap W_n|=o(|W_n|),
\end{equation}
then
\begin{equation} \label{lastlem}
\frac{1}{|W_n|}\int_{W_n^k}f(x)\der\nu^k(x)\cvn\int_{\X^{k-1}}f(0,x_2,\cdots,x_k)\der\nu(x_2)\cdots\der\nu(x_k).
\end{equation}
\end{lem}
\begin{proof}
We write $W_n\omoins r$ for the set $W_n\backslash (\partial W_n\oplus r)$ of points in $W_n$ at distance at least $r$ from the boundary of $W_n$. Since $f$ is translation invariant then the right term in \eqref{lastlem} is equal to
$$\frac{1}{|W_n|}\int_{W_n\times\X^{k-1}}f(x)\der\nu^k(x).$$
As a consequence,
\begin{align*}
&\left|\int_{\X^{k-1}}f(0,x_2,\cdots,x_k)\der\nu(x_2)\cdots\der\nu(x_k)-\frac{1}{|W_n|}\int_{W_n^k}f(x)\der\nu^k(x)\right|\\
&=\frac{1}{|W_n|}\left|\int_{W_n\times (\X^{k-1}\backslash W_n^{k-1})} f(x)\der\nu^k(x)\right|\\
&=\frac{1}{|W_n|}\left|\int_{W_n\omoins r}\left(\int_{\X^{k-1}\backslash W_n^{k-1}}f(x)\der\nu(x_2)\cdots\der\nu(x_k)\right)\der\nu(x_1)\right.\\
&\hspace{4cm}\left.+\frac{1}{|W_n|}\left(\int_{(\partial W_n\oplus r)\cap W_n}\int_{\X^{k-1}\backslash W_n^{k-1}}f(x)\der\nu(x_2)\cdots\der\nu(x_k)\right)\der\nu(x_1)\right|\\
&\leq \frac{1}{|W_n|}\int_{W_n\omoins r}\left (\int_{\X^{k-1}}|f(0,y)|\cara{\{\forall i,~\|y_i\|>r\}}\der\nu^{k-1}(y)\right )\der\nu(x)\\
&\hspace{5cm}+\frac{1}{|W_n|}\int_{(\partial W_n\oplus r)\cap W_n}\left (\int_{\X^{k-1}}|f(0,y)|\der \nu^{k-1}(y)\right )\der \nu(x)\\
&\leq \int_{(\mathcal{B}(0,r)^C)^{k-1}}|f(0,y)|\der \nu^{k-1}(y)+\frac{|(\partial W_n\oplus r)\cap W_n|}{|W_n|}\|f(0,.)\|_{L^1}.\qedhere
\end{align*}
\end{proof}

\begin{prop} \label{rhocool}
Let $X$ be a DPP with Bessel-type kernel $K^{\rho,\alpha}_0$, as defined in Table~\ref{parametric models}, observed on a window $W\subset\R^d$. Recall that $\rho_{\max}$, given  in Table~\ref{parametric models},  is the upper bound of $\rho$ for which $X$ is well-defined.
Then, for all $\alpha>0$ such that $N(W)/|W|\leq \rho_{\max}$,
\begin{equation} \label{argmax}
\argmax_{0\leq \rho\leq \rho_{\max}} \tilde l(\rho,\alpha|X)=\left\{\frac{N(W)}{|W|}\right\}.
\end{equation}
\end{prop}

\begin{proof}
By noticing that $\rho_{\max}$ is the volume of the $d$-dimensional ball with radius $\sqrt{d/(2\pi^2\alpha^2)}$, we get from the expression of $\hat K^{\rho,\alpha}_0$  in Table~\ref{parametric models} that
$$\int_{\R^d}\log(1-\hat K^{\rho,\alpha}_0(x))\der x=\rho_{\max}\log(1-\rho/\rho_{\max}).$$
Moreover, $L^{\rho,\alpha}_0$ can be written as $\rho F^\alpha / (1-\rho/\rho_{\max})$, where $F^\alpha$ is a function not depending on $\rho$ (see Table~\ref{L0 models}). Therefore, $\logdet(L^{\rho,\alpha}_0[X\cap W])$ can be expressed as the sum of
$$N(W)\log\left(\frac{\rho}{1-\rho/\rho_{\max}}\right)$$
and an expression not depending on $\rho$. As a consequence, $\tilde l(\rho,\alpha|X)$ is twice differentiable with respect to $\rho$ with derivative
$$\frac{-1}{1-\rho/\rho_{\max}}+\frac{N(W)}{|W|\rho(1-\rho/\rho_{\max})}.$$
It is easy to see that this expression vanishes only when $\rho=N(W)/|W|$ with the second derivative being negative at this point, concluding the proof.
\end{proof}

\bibliographystyle{plain}
\bibliography{BigRef}

\end{document}